\documentclass[10pt,twoside]{amsart}
\usepackage{amsmath,amssymb}
\usepackage{amsfonts}
\usepackage{amsthm,amscd}
\usepackage{amstext}
\usepackage{color}
\usepackage{latexsym}
\usepackage{amscd,graphics}

\begin{document}
\newcommand{\comment}[1]{\marginpar{\footnotesize #1}} 

\newtheorem{proposition}{Proposition}[section]
\newtheorem{lemma}[proposition]{Lemma}
\newtheorem{sublemma}[proposition]{Sublemma}
\newtheorem{theorem}[proposition]{Theorem}

\newtheorem{maintheorem}{Main Theorem}
\newtheorem{corollary}[proposition]{Corollary}

\newtheorem{ex}[proposition]{Example}

\theoremstyle{remark}

\newtheorem{remark}[proposition]{Remark}

\theoremstyle{definition}
\newtheorem{definition}[proposition]{Definition}
\def\BB{\mathcal{B}}
\def\FF{\mathcal{F}}
\def\GG{\mathcal{G}}
\def\II{\mathcal{I}}
\def\JJ{\mathcal{J}}
\def\LL{\mathcal{L}}
\def \PP{\mathcal {P}}
\def \QQ{\mathcal {Q}}
\def\PPP{\mathbb{P}}

\def\real{{\mathbb R}}
\def\rational{{\mathbb Q}}
\def\natural{{\mathbb N}}
\def\integer{{\mathbb Z}}

\def\B{\mathcal{B}}
\def\C{\mathcal{C}}
\def\O{\mathcal{C}}
\def\P{\mathcal{P}}
\def\S{\mathcal{U}}
\def\U{\mathcal{U}}

\def\const{\operatorname{const}}
\def\closure{\operatorname{closure}}
\def\dist{\operatorname{dist}}
\def\esssup{\operatorname{ess\ sup}}
\def\essinf{\operatorname{ess\ inf}}
\def\inte{\operatorname{int}}
\def\Lip{\operatorname{Lip}}
\def\max{\operatorname{max}}
\def\min{\operatorname{min}}
\def\mod{\operatorname{mod}}
\def\osc{\operatorname{osc}}
\def\sign{\operatorname{sign}}
\def\supp{\operatorname{supp}}

\def\al{\alpha}
\def\be{\beta}
\def\vep{\varepsilon}
\def\th{\theta}
\def\om{\omega}
\def\la{\lambda}
\def\La{\Lambda}
\def\ga{\gamma}
\def\ka{\kappa}

\def\ome{\tilde{\om}}
\def\a{a_J}
\def\I{I}
\def\i{[0,\delta]}
\def\ii{[0,\vep]}
\def\l{l_0}

\title[LIL and ASIP for one-parameter families of interval maps]
{Law of iterated logarithm and invariance principle for one-parameter families of interval maps}
\author{Daniel Schnellmann}
\address{DMA, UMR 8553, \'Ecole Normale Sup\'erieure,  75005 Paris, France}
\email{daniel.schnellmann@gmail.com}
\date{\today}

\begin{abstract}
We show that for almost every map in a transversal one-parameter family of piecewise expanding 
unimodal maps the Birkhoff sum of suitable observables along the forward orbit of the turning point 
satisfies the law of iterated logarithm. This result will follow from 
an almost sure invariance principle for the Birkhoff sum, 
as a function on the parameter space. Furthermore, we obtain a 
similar result for general one-parameter families 
of piecewise expanding maps on the interval.
\end{abstract}

\thanks{The author is grateful to Viviane Baladi for encouraging him to work on this question. 
This work was accomplished at DMA, ENS, Paris. The author was supported 
by the Swiss National Science Foundation.}
\maketitle

\section{Introduction}
In this introduction we consider only piecewise expanding unimodal maps. However, 
all the following results can be extended to more general families of piecewise 
expanding interval maps (see Section~\ref{s.uniform}).
We call a map $T:[0,1]\to[0,1]$ a {\em piecewise expanding unimodal map} or {\em tent map} 
if it is continuous and if there exists a turning point 
$c\in(0,1)$ such that $T|_{[0,c]}$ and $T|_{[c,1]}$ are $C^{1+\alpha}$, $\|1/T'\|_\infty<1$
and $\|T'\|_\infty<\infty$, 
and $T(1)=T(0)=0$. 
We assume that $T$ is {\em mixing}, i.e., it is topologically mixing in the interval $[T^2(c),T(c)]$.
Let $\mu$ denote the unique (hence ergodic) absolutely continuous invariant probability 
measure (acip) for $T$. 
By Birkhoff's ergodic theorem, $\mu$ almost every (or in this case also Lebesgue 
almost every) point $x\in[0,1]$ is {\em typical} for $\mu$, i.e., 
\begin{equation}
\label{eq.typical}
\lim_{n\to\infty}\frac1n\sum_{i=1}^n\varphi(T^i(x))=\int_0^1\varphi d\mu\,,\qquad\forall \varphi\in C^0\,.
\end{equation}
A natural question is how fast this convergence takes place. In order to answer this 
question one has to take a smaller set of observables: By \cite{krengel} and \cite{djr}, 
for any sequence $\alpha_n$ such that $\lim_{n\to\infty}\alpha_n=\infty$, there is a 
dense $G_\delta$ set in $C^0$ such that for all $\varphi$ in this set one has
$$
\lim_{n\to\infty}\alpha_n\Big|\frac1n\sum_{i=1}^n\varphi(T^i(x))-\int_0^1\varphi d\mu\Big|=\infty\,.
$$
A suitable set of observables for which the question about the speed of convergence makes 
sense is for example the set of H\"older continuous functions (or more generally the 
set of functions of generalised bounded variation; see Definition~\ref{d.generalizedbv} below). 
For $\varphi$ H\"older, set 
\begin{equation}
\label{eq.sigma}
\sigma(\varphi)^2:=\int_0^1\Big(\varphi-\int\varphi d\mu\Big)^2d\mu
+2\sum_{i>0}\int_0^1\Big(\varphi-\int\varphi d\mu\Big)\Big(\varphi-\int\varphi d\mu\Big)\circ T^id\mu\,.
\end{equation}
Since we have exponential decay of correlation (see, e.g., Proposition~\ref{p.udc} below), 
$\sigma(\varphi)$ is finite and since we can 
write $\sigma(\varphi)=\lim_{n\to\infty}n^{-1}Var(S_n)$, where $S_n$ is the $n$-th Birkhoff sum, 
we see that $\sigma(\varphi)\ge0$. If $\sigma(\varphi)=0$, then $\varphi$ is a co-boundary and 
there exists an $L^1$ function $\psi$ so that $\varphi=\psi\circ T-\psi$ almost surely. 
Henceforth, we exclude this (degenerate) case, i.e., we will always assume that $\sigma(\varphi)>0$. 
Turning back to the question about the speed of convergence of \eqref{eq.typical}, it is shown in 
\cite{hk} that if we restrict ourself to the set of H\"older continuous observables $\varphi$ 
then the {\em law of iterated logarithm (LIL)} holds: For a.e. $x\in[0,1]$, we have
\begin{equation}
\label{eq.lil}
\limsup_{n\to\infty}\frac1{\sqrt{2n\log\log n}}\sum_{i=1}^n\Big(\varphi(T^i(x))-\int\varphi d\mu\Big)
=\sigma(\varphi)\,.
\end{equation}

For tent maps $T$ the turning point $c$ is of particular dynamical interest. 
A lot of information about the dynamics of $T$ is contained in the forward orbit of $c$, and
it is natural to ask if \eqref{eq.lil} holds when we take $x=c$. 
For a recent work where the assumption that the turning point satisfies the 
LIL is crucial, see \cite{bms}. 
 However, even if 
we know that \eqref{eq.typical} and \eqref{eq.lil} hold for a.e. point $x$, it is a very difficult 
question to say wether they hold for a particular point $x$. So instead of asking for 
the LIL for $c$ for a single tent map $T$, we perturb this map by a 
one-parameter family of tent maps and ask if the LIL for $c$ holds 
for almost every map in this family. 
Let $T_a$, $a\in[0,1]$, be a one-parameter family of piecewise expanding unimodal maps 
through $T=T_0$.
We make some natural regularity assumptions on 
the parameter dependency as, e.g., the turning point $c_a$ is 
Lipschitz continuous in $a$ and if $J\subset[0,1]$ is an interval on which $x\neq c_a$, 
then $a\mapsto T_a(x)$ is $C^{1+\alpha}$ on $J$ (for the precise conditions 
we refer to the beginning of Section~\ref{s.uniform}). 
Of course in order that the question of this paragraph makes sense we have to exclude 
trivial one-parameter families as for example the constant one or  families for 
which the turning point is eventually mapped to a periodic point for all parameters. 
The right condition here is {\em transversality} which is a common non-degeneracy condition 
for one-parameter families of interval maps (see, e.g., 
\cite{tsujii}, \cite{ALM}, \cite{Levin}, \cite{BS}, \cite{GS}, \cite{s2}, 
\cite{BBS} for previous occurrences of this condition in 
the literature). 
We say that the family $T_a$ is {\em transversal} at $T_0$ if 
there exists a constant $C\ge1$ such that
\begin{equation}
\label{eq.transversal}
C^{-1}\le\left|\frac{\partial_aT_a^j(c_a)|_{a=0}}{(\partial_xT_0^{j-1})(T_0(c_0))}\right|\le C\,,
\quad\forall\ \text{$j$ large.}
\end{equation}
(If $c_0$ is periodic for $T_0$, then we take one-sided derivatives.)
The transversality condition says that the $a$-derivative along 
the postcritical orbit is comparable to its $x$-derivate. Since the 
$x$-derivative is growing exponentially fast, this implies that if we change the 
parameter the dynamics of the corresponding map will change fast which makes 
it then possible to study the generic behaviour of the postcritical orbit.
If the family $T_a$ is {\em transversal} at $T_0$, then it is shown in \cite{s2} that 
for a.e. parameter $a$ close to $0$ the turning point $c_a$ is typical for the acip $\mu_a$ 
(for related results see \cite{bruin}, \cite{schmeling}, and \cite{fp}).
Given almost sure typicality of the turning point we can now ask for the speed 
of convergence of \eqref{eq.typical} in this setting.

The main result of this paper can be stated as follows (see also Theorem~\ref{t.tenttransversal} 
in Section~\ref{s.examples} below). To the best of the authors knowledge, it 
is the first result which treats the question of a LIL for a {\em specific} point for a.e. parameter 
in a one-parameter family of dynamical systems. 
Recall the notation $\sigma$ in \eqref{eq.sigma}. We 
will use the notation $\sigma_a$ when considering the map $T_a$.

\begin{theorem}
\label{t.lil}
Assume that $T_0$ is mixing, its turning point $c_0$ is not periodic, and the family $T_a$ is transversal at $T_0$. 
If $\varphi$ is H\"older 
and $\sigma_0(\varphi)>0$, then there exists $\epsilon>0$ 
such that for almost every $a\in[0,\epsilon]$ the 
turning point $c_a$ satisfies the LIL 
for the function $\varphi$ under the map $T_a$, i.e., 
$$
\limsup_{n\to\infty}\frac1{\sqrt{2n\log\log n}}\sum_{i=1}^n
\Big(\varphi(T_a^i(c_a))-\int\varphi d\mu_a\Big)
=\sigma_a(\varphi)\,.
$$
\end{theorem}

In order to prove Theorem~\ref{t.lil}, we will show a stronger property, the so called {\em almost sure 
invariance principle (ASIP)}, for the turning point. 
We say that the functions $\xi_i:[0,\epsilon]\to\real$, $i\ge1$, satisfy the ASIP with error 
exponent $\gamma<1/2$ if 
there exists a probability space supporting a Brownian motion $W$ and a sequence of 
variables $\eta_i$, $i\ge1$, such that
\begin{itemize}
\item[(i)]
$\{\xi_i\}_{i\ge1}$ and $\{\eta_i\}_{i\ge1}$ have the same distribution;
\item[(ii)]
almost surely as $n\to\infty$, 
$$
\left|W(n)-\sum_{i=1}^{n}\eta_i\right|=O(n^\gamma)\,.
$$
\end{itemize}

The following corollary is shown, e.g., in \cite{ps}. For other implications of the ASIP
we refer to \cite{hh}.

\begin{corollary}
\label{c.lil}
If the functions $\xi_i$ 
satisfy the ASIP then they satisfy also the LIL and the central limit theorem. More 
precisely, if $\sigma^2$ is the variance of the related Brownian motion, then 
$$
\limsup_{n\to\infty}\frac1{\sqrt{2n\log\log n}}\sum_{i=1}^n
\xi_i(a)
=\sigma\,,\qquad\text{for a.e. }a\in[0,\epsilon]\,,
$$
and, for all $t\in\real$,
$$
\lim_{n\to\infty}m\Big(\Big\{a\in[0,\epsilon]\mid\frac1{\sigma\sqrt{n}}\sum_{i=1}^n\xi(a)\le t\Big\}\Big)
=\frac1{\sqrt{2\pi}}\int_{-\infty}^te^{-s^2/2}ds\,,
$$
where $m$ denotes the Lebesgue measure.
\end{corollary}

For $\varphi$ H\"older and such that $\sigma_0(\varphi)>0$, 
for $i\ge1$ and $a$ small, set 
\begin{equation}
\label{eq.varphia}
\varphi_a(x):=\frac1{\sigma_a(\varphi)}
\Big(\varphi(x)-\int_0^1\varphi d\mu_a\Big). 
\end{equation}
Lemma~\ref{l.varregularity} below guarantees\footnote{In order that condition~(II) in 
Lemma~\ref{l.varregularity} is satisfied,
we assume that $T_0$ is mixing and $c_0$ is not periodic (see proof of 
Theorem~\ref{t.tent}).} that
$a\mapsto\sigma_a(\varphi)$ is continuous at $0$ and, hence, 
the function $\varphi_a$ is well-defined for $a$ sufficiently 
close to $0$. Due to this normalisation we have
\begin{equation}
\label{eq.varphinormalised}
\sigma_a(\varphi_a)=1\,,\qquad\text{and}\quad\int\varphi_ad\mu_a=0\,,
\end{equation}
for all $a$ sufficiently close to $0$.
We are going to show an ASIP for the functions
\begin{equation}
\label{eq.xi}
\xi_i(a):=\varphi_a(T_a^i(x))\,,\qquad i\ge1\,.
\end{equation}

\begin{theorem}
\label{t.asip}
Assume that $T_0$ is mixing, $c_0$ is not periodic, and the family $T_a$ is transversal at $T_0$. 
Then there exists $\epsilon>0$ such that the functions $\xi_i:[0,\epsilon]\to\real$, 
$i\ge1$, satisfy the ASIP for all error exponents $\gamma>2/5$.
\end{theorem}

\begin{remark}
Theorem~\ref{t.asip} and, hence, Theorem~\ref{t.lil} hold 
also if $c_0$ is periodic and $T_0$ has a sufficiently high expansion (see 
Theorem~\ref{t.tenttransversal} in Section~\ref{s.examples} below). 
Because of the normalisation in the definition of the $\xi_i$'s, the variance 
of the Brownian motion in the ASIP is equal to $1$. 
For a comment on the optimality of the error exponent $\gamma$ see the beginning of 
Section~\ref{s.asip}.
\end{remark}

Regarding the proof of Theorem~\ref{t.asip} we go along a classical method form 
probability theory which consists in writing 
the Birkhoff sum approximatively as a sum of blocks of polynomial size, then in 
approximating these blocks by a martingale difference sequence, and finally in applying 
Skorokhod's representation theorem which provides a link between a martingale and a 
Brownian motion. This strategy is illustrated on 
many examples in Philipp and Stout \cite{ps}.
More precisely, we go along the approach in \cite[Section~3]{ps}. 
The 'usual' application of \cite{ps} in dynamical systems refers to \cite[Section~7]{ps} 
(see, e.g., \cite{hk}, \cite{dp}, and \cite{mn}). The key property here is a strong mixing condition 
which we do not have in our setting since loosely speaking the $\xi_i$'s are not iterations of a
fixed map. However, we can more or less replace this strong mixing condition by 
{\em uniformity of constants} in the Lasota-Yorke inequality for the family $T_a$ (see 
condition~(II) in Section~\ref{s.uniform}). By Keller-Liverani, 
we have then uniformity of constants for the exponential decay of correlation 
(see Proposition~\ref{p.udc}). This in turn can 
be used to show a certain exponential decay of correlation for the maps $\xi_i$'s (see, e.g., 
the proof of Proposition~\ref{p.mainestimate}) from which 
we are able to deduce similar estimates as in \cite[Section~3]{ps}. 
In the recent work \cite{gouezel}, Gou\"ezel uses spectral methods to show an almost 
sure invariance principle. His method is very powerful and it provides very good 
error estimates. However, we didn't find an easy way to apply these spectral techniques 
to our setting.

We would like to highlight that the main technical novelty or difficulty of this paper is 
to treat processes which are (at least ``locally'') close to processes generated by a dynamical system 
but for which there is no underlying invariant measure. Hence, various tools from 
ergodic theory cannot be applied directly. This explains the rather technical nature 
of this paper. The following example by Erd\"os and 
Fortet (see \cite{kac}, p. 646) shows how careful one should be when 
one wants to show an ASIP for a process which is not but very close to a process 
generated by a dynamical system: Let $\varphi(x)=\cos(2\pi x)+\cos(4\pi x)$ and 
consider the sequence $\xi_i(x)=\varphi(2^ix)$, $i\ge1$. $\xi_i$ is a process generated by 
the doubling map $x\mapsto 2x\mod 1$. It is straightforward to check that $\sigma(\varphi)>0$ 
and, for instance by the above cited ``dynamical'' paper \cite{hk}, 
it follows that the process $x_i$ satisfies the ASIP. However, if we change the process just 
slightly and consider instead $\xi_i(x)=\varphi((2^i-1)x)$ then, surprisingly, 
this new process does not satisfy anymore the central limit theorem (and, thus, 
not either the ASIP). 

As mentioned in the beginning of this section, the above presented results for tent maps hold 
for more general piecewise expanding maps on the interval. First, it is not essential to take 
the turning point as the point of interest. Any other point works fine as long as the $a$- and 
$x$-derivatives along its forward orbit are comparable. If we consider other piecewise expanding 
maps on the interval than tent maps, then we have to add two more conditions. 
The first one is to have uniform constants in the Lasota-Yorke inequality (see condition~(II) 
in Section~\ref{ss.main}). 
This is a natural condition when applying perturbation theory. The second condition 
(see condition~(III) in Section~\ref{ss.main}) is 
a bit more technical but satisfied for many one-parameter families, as it is shown in 
Section~\ref{s.examples}. In Section~\ref{s.examples} we mention also how to 
apply the main result of this paper, Theorem~\ref{t.main}, to obtain 
almost sure typicality results similar to the ones in \cite{s2} but under alternative conditions
(see Theorem~\ref{t.typical}).

The present paper deals exclusively with maps which are uniformly hyperbolic. 
It is a natural question if we can obtain a similar result in a non-uniformly 
hyperbolic setting. An interesting candidate for this question is the quadratic family 
$f_a(x)=ax(1-x)$ with parameter $a\in(0,4]$. Does the critical point $c=1/2$ satisfy 
the LIL for Lebesgue almost every Collet-Eckmann (CE) map for sufficiently smooth observables?
Despite a vast variety of results about the quadratic family, this question seems still to be
unsolved. To start with one should maybe content oneself with
finding a positive Lebesgue measure set 
of CE parameters such that the critical points of the 
corresponding CE maps satisfy the LIL. 
Almost sure typicality of the critical point is known: By Avila and Moreira \cite{avilamoreiratypical}, 
the critical point for Lebesgue almost every CE map $f_a$ in the 
quadratic family is typical for its SRB measure $\mu_a$. (For the subset of CE parameters 
considered by Benedicks and Carleson this result was shown in \cite{BC1}.)
An important ingredient in an attempt to find a positive measure set for which the turning point 
satisfies the LIL should be uniformity of constants in the set of CE parameters which one considers. 
For this one could 
follow the ``start-up procedure'' in Benedicks and Carleson \cite{BC1} 
which yields, in addition to uniformity of constants, at each step nice ``Markov partitions'' 
on the parameter space. On the partition elements of these Markov partitions, which are intervals, 
one should be able to define functions $\xi_i(a)$ as in \eqref{eq.xi}. (Observe that at each step one 
excludes parameter intervals from the previous Markov partition and, finally, one ends up 
with a Cantor set of positive Lebesgue measure.)
In \cite{mn} where the ASIP is shown for a fixed CE map, they use a tower construction 
to get an induced system with uniform hyperbolicity 
where more or less a straight forward application of \cite[Section~7]{ps} implies 
an ASIP which projects then down to the ASIP for the original CE map. Since 
in the parameter space one has to exclude an open and dense set of regular 
parameters, the ``start-up procedure'' in \cite{BC1} might provide a way to replace 
this tower construction in \cite{mn} when one deals only with one single CE map.

The paper is organised as follows. In Section~\ref{s.uniform}, we formulate 
a general model and give the main notations 
for the one-parameter families of piecewise expanding maps considered in this paper. This is 
followed by the main statement. Section~\ref{s.examples} 
contains examples of one-parameter families, such as families of tent maps, 
to which the result of this paper applies. Section~\ref{s.prel} deals with elementary facts 
as distortion estimates, uniform exponential decay of correlations, and the regularity of 
$a\mapsto\sigma_a$. Section~\ref{s.switch} and \ref{s.asip} are dedicated to the proof of the main 
statement, i.e., the proof of 
an almost sure invariance principle.

\section{Main statement}
\label{s.uniform}
We begin this section with an introduction of the basic notation and a formulation of 
a suitable model for 
one-parameter families of {\em piecewise expanding maps} of the unit interval. 
A map $T:[0,1]\to[0,1]$ will be called {\em piecewise $C^{1+\alpha}$\/}, $0<\alpha\le1$, 
if there exists a partition 
$0=b_0<b_1<...<b_p=1$ of the unit interval such that for each $1\le k\le p$ the restriction 
of $T$ to the open interval $(b_{k-1},b_k)$ is a $C^{1+\alpha}$ function. 
Let $T_a:[0,1]\to[0,1]$, 
$a\in[0,1]$, be a one-parameter family of piecewise $C^{1+\alpha}$ maps 
and let $0=b_0(a)<b_1(a)<...<b_{p(a)}(a)=1$ be the partition of the unit interval 
associated to $T_a$. 
We assume that the 
H\"older constants are uniform in $a$, i.e., there exist $0<\alpha\le1$ and a constant $C$ so that 
\begin{equation}
\label{eq.holder}
|T_a'(x)-T_a'(y)|\le C|x-y|^\alpha\,,\qquad \forall x,y\in (b_{k-1}(a),b_k(a))\ \ \text{and }\ \forall a\in[0,1]\,.
\end{equation}
We assume that the maps are uniformly {\em expanding}, i.e., we assume that 
there are real numbers $1<\la\le\La<\infty$ such that for every $a\in[0,1]$, 
\begin{equation}
\label{eq.la}
\|1/T_a'\|_\infty\le\lambda^{-1}\,,\qquad\text{and}\qquad\|T_a'\|_\infty\leq\La\,.
\end{equation}  

\begin{remark}
Regarding the framework in \cite{keller} and \cite{saussol}, it would be natural 
to skip the assumption that $\|T_a'\|_\infty\leq\La$ and to replace the 
requirement that $x\mapsto T_a'(x)$ is piecewise $\alpha$-H\"older by the requirement that 
$x\mapsto1/T_a'(x)$ is piecewise $\alpha$-H\"older. However, since our analysis 
on the parameter space seems to require some specific distortion estimates (see 
Lemma~\ref{l.distortion} below), we do not know how 
to make this improvement in our setting. For instance having in mind one-parameter families of 
one-dimensional Lorenz maps, it might be interesting to investigate such a more general setting.
\end{remark}

We make the following natural assumptions on the parameter dependence. 

\begin{itemize}
\item[(i)] 
The number of monotonicity intervals for the $T_a$'s is constant, i.e., $p(a)\equiv p_0$, 
and the partition points $b_k(a)$, $0\le k\le p_0$, are Lipschitz continuous on $[0,1]$. 
It follows that there is a constant $\delta_0>0$ such that 
$$
b_k(a)-b_{k-1}(a)\ge\delta_0,
$$
for all $1\le k\le p_0$ and $a\in[0,1]$.
\item[(ii)]
If $x\in[0,1]$ and $J\subset[0,1]$ is a parameter interval 
such that $b_k(a)\neq x$, for all $a\in J$ and $0\le k\le p_0$, then $a\mapsto T_a(x)$ is 
$C^{1+\alpha}$ and $a\mapsto \partial_xT_a(x)$ is $\alpha$-H\"older 
where the H\"older constants are independent on $x$. 
Further, the maps $x\mapsto\partial_a T_a(x)$, $x\in(b_{k-1}(a),b_k(a))$, $1\le k\le p_0$,  are 
$\alpha$-H\"older continuous (where the H\"older constants are uniform in $a$). 
\end{itemize}

In order to obtain an acip, we refer to a paper by 
G. Keller \cite{keller} (see Theorems~3.3 and 3.5 therein) 
who extended the results in \cite{lasota} on piecewise expanding 
$C^2$ maps to a broader class of maps containing also piecewise expanding $C^{1+\alpha}$ maps: 
For a fixed $a\in[0,1]$ there exists a finite number of ergodic acip for $T_a$. 
Further, by \cite{liyorke} combined with the remark in \cite{wong2} after Definition~4 on page 514 
(regarding property~(III) therein cf. also \cite[Proposition~ 5.1]{saussol}), 
there exist at most $p_0-1$ ergodic acip and the support of an ergodic acip is 
a finite union of intervals. 
Since we are always interested in only one ergodic acip, we can 
without loss of generality assume that for each $T_a$, $a\in[0,1]$, 
there is a unique (hence ergodic) acip which we denote by $\mu_a$.
Let $K(a)=\supp(\mu_a)$. 
We say that $T_a$ is {\em mixing} if it is topologically mixing on $K(a)$.
For $a\in[0,1]$, let 
$\{D_1(a),...,D_{p_1(a)}(a)\}$ be the connected components of 
$K(a)\setminus\{b_0(a),...,b_{p_0}(a)\}$, 
i.e., the $D_k(a)$'s are the monotonicity intervals for $T_a:K(a)\to K(a)$. We assume the following.
\begin{itemize}
\item[(iii)]
The number of $D_k(a)$'s is constant in $a$, i.e., $p_1(a)\equiv p_1$ for all $a\in[0,1]$. 
The boundary points of 
$D_k(a)$, $1\le k\le p_1$, are $\alpha$-H\"older continuous in $a$.
\end{itemize}

\subsection{Partitions}
\label{ss.partitions}
For a fixed parameter value $a\in[0,1]$, we denote by $\P_j(a)$, $j\ge1$, 
the partition on the dynamical interval 
consisting of the maximal open intervals of smooth monotonicity for the map $T_a^j: K(a)\to K(a)$. 
More precisely, $\P_j(a)$ denotes the set of open intervals $\om\subset K(a)$ such that 
$T_a^j:\om\to K(a)$ is $C^{1+\alpha}$ and $\om$ is maximal, i.e., for every other open interval $\ome\subset K(a)$ 
with $\om\subsetneq\ome$, $T_a^j:\ome\to K(a)$ is no longer $C^{1+\alpha}$. 
Clearly, the elements of $\P_1(a)$ are the interior of the intervals $D_k(a)$, $1\le k\le p_1$.

We will define similar partitions on the parameter interval $[0,1]$. 
Let $x_0:[0,1]\to[0,1]$ be a $C^{1+\alpha}$ map from the parameter interval $[0,1]$ 
into the dynamical interval $[0,1]$ where we assume that
\begin{equation}
\label{eq.convenient}
x_0(a)\in K(a)\setminus\{b_0(a),...,b_{p_0}(a)\}\,,\qquad \forall a\in(0,1)\,.
\end{equation}
The points $x_0(a)$, $a\in[0,1]$, are the points of interest in this paper, i.e., we 
are interested in the properties of the forward orbit of these points under $T_a$. 
The assumption \eqref{eq.convenient} is only for convenience and it helps to make 
the partitions $\PP_j$ below well-defined. (If a map $x_0$ does not satisfy \eqref{eq.convenient}, 
then combining the fact that Lebesgue a.e. point $x\in[0,1]$ is eventually mapped into $K(a)$ 
under $T_a$ with the transversality condition~(I) below, one can derive that \eqref{eq.convenient} 
is satisfied for some iteration of $x_0$ restricted to some smaller intervals located around 
$a=0$.)
The forward orbit of a point $x_0(a)$ under the map $T_a$ 
we denote as 
$$
x_j(a):=T_a^j(x_0(a)),\qquad\quad j\ge0.
$$
Observe that by assumption $x_j(a)\in K(a)$, for all $j\ge0$ and $a\in[0,1]$.

\begin{remark}
\label{rem.critical}
Since a lot of information for the dynamics of $T_a$ is contained in the forward orbits of the 
partition points $b_k(a)$, $0\le k\le p_0$, an interesting choice of the map $x_0$ is
$$
x_0(a)=\lim_{\begin{subarray}{c}
           x\to b_k(a)\pm 
           \end{subarray}}
     T_a(x).
$$
For example, in the case of tent maps we choose $x_0(a)=T_a^{j_0}(c_a)$, for $j_0$ sufficiently 
large (see Theorem~\ref{t.tenttransversal} below).
\end{remark}

Let $J\subset[0,1]$ be an interval.
By $\P_j|J$, $j\ge1$, we denote the partition consisting 
of all open intervals $\om$ in $J$ such that 
for each $0\le i<j$, $x_i(a)\in K(a)\setminus\{b_0(a),...,b_{p_0}(a)\}$, for all $a\in\om$, and such 
that $\om$ is maximal, 
i.e., for every other open interval $\ome\subset J$ with $\om\subsetneq\ome$, there exist 
$a\in\ome$ and $0\le i<j$ such that $x_i(a)\in\{b_0(a),...,b_{p_0}(a)\}$. 
Observe that this partition might be empty which is, e.g., the case when $x_1(a)$ is equal 
to a boundary point $b_k(a)$ for all $a\in[0,1]$.
However, such trivial situations (around $a=0$) 
are excluded by the transversality condition~(I)
formulated in the next Section~\ref{ss.main}. Knowing 
that condition~(I) is satisfied, then the partition $\P_j|J$, $j\ge1$, around $a=0$ 
can be thought of as the set of the 
(maximal) intervals of smooth monotonicity for $x_j:J\to[0,1]$
(cf. Lemma~\ref{l.transversality} below).
We set $\P_0|J=J$. 
Finally, in view of condition~(I) below, observe that if a parameter 
$a\in[0,1]$ is contained in an element of 
$\P_j|[0,1]$, $j\ge1$, then also the point $x_0(a)$ is contained in an element of 
$\P_j(a)$ which implies 
that $T_a^j$ is differentiable in $x_0(a)$.

\subsection{Main statement}
\label{ss.main}
We put two conditions on our sequence of maps $x_j$, $j\ge0$, around $a=0$. 
The first one (see condition~(I) below) is a common {\em transversality} condition 
for one-parameter families of interval maps which was already mentioned in 
the introduction. 
The second one (see condition~(III) below) is more technical. It 
is used for controlling the measure of the set of partition elements with a too small image. 
Further, in order to apply perturbation results we require that 
we have uniform constants in the Lasota-Yorke inequalities for 
the different maps in the family (see condition~(II) below).
This condition does not depend on the choice of the map $x_0$. 
Even if condition~(III) is quite technical, assuming that conditions~(I) and (II) hold, 
it is satisfied by many important one-parameter families of 
piecewise expanding maps, see Section~\ref{s.examples}. 
(Even if we suspect so, it is not clear to us if 
in general the transversality condition~(I), possibly together with condition~(II) and/or 
some weaker other conditions, implies 
condition~(III).) 

The transversality condition~(I) requires that the derivatives of $x_j$ and
$T_a^j$ at $a=0$ are comparable. This is the very basic assumption in this
paper. It says that locally the behaviour of the maps $x_j$ are comparable to the 
behaviour of the maps $T_a^j$. Since the LIL holds for the maps $T_a^j$ 
one can therefore hope to obtain similar properties for the maps $x_j$.
Of course, in order to have transversality the choice of the map $x_0:[0,1]\to[0,1]$ plays an
important role. If, e.g., for every parameter $a\in[0,1]$, $x_0(a)$ is a
periodic point for the map $T_a$, then $x_j$ will have bounded
derivatives and the dynamics of $x_j$ is completely different from the
dynamics of $T_a$. Henceforth, we will use the notations 
$$
T_a'(x)=\partial_xT_a(x)\qquad\text{and}\qquad x_j'(a)=\partial_a x_j(a),\quad j\ge1.
$$ 

\begin{itemize}
\item[(I)]
The right-derivatives $x_j'(0+)$, $j\ge1$, of $x_j$ in $0$ exist and
there is a constant $C\ge1$ so that  
$$
\frac{1}{C}\le\left|\frac{x_j'(0+)}{(T_0^j)'(x_0(0+))}\right|\le C\,,\qquad\forall j\ge1\,.
$$
Further, for each $j\ge1$, there exists a neighbourhood $V\subset[0,1]$ of $0$ 
so that for all $a\in V\setminus0$ and all $0\le i<j$, we have $x_i(a)\notin\{b_0(a),...,b_{p_0}(a)\}$.
\end{itemize}

\begin{remark}
Looking at the proof of the following Lemma~\ref{l.transversality}, one can derive that 
condition~(I) is satisfied if
$$ 
|x_0'(0+)|\ge\frac{\sup_{a\in[0,1]}\sup_{x\in K(0)}|\partial_aT_a(x)|_{a=0}|}{\la-1}+2L+1,
$$
where $L$ is the Lipschitz constant of the partition points $b_0(a),...,b_{p_0}(a)$.
In other words, as soon as the initial derivative is sufficiently large we have transversality 
which makes it easy to verify this property numerically.
\end{remark}

The following lemma ensures that if condition~(I) holds then we can compare 
the $a$- and the $x$-derivatives along the forward orbit of $x_0(a)$ on an entire,
sufficiently small interval around $0$. 
Its proof is given at the end of this section.

\begin{lemma}
\label{l.transversality}
Assume that the family $T_a$ satisfies condition~(I). Then, 
there exists $\epsilon>0$ and a constant $C\ge1$ 
so that for $\om\in\PP_j|[0,\epsilon]$, $j\ge1$, we have
$$
\frac{1}{C}\le\left|\frac{x_j'(a)}{(T_a^j)'(x_0(a))}\right|\le C\,,\qquad\forall a\in\om\,.
$$
Furthermore, for each $j\ge1$, 
the number of $a\in[0,\epsilon]$ which are not contained in any element 
$\om\in\PP_j|[0,\epsilon]$ is finite.
\end{lemma}

Apart from transversality we require also to have uniform constants in the Lasota-Yorke 
inequality. Let $\LL_a:L^1([0,1])\to L^1([0,1])$ be the ordinary (Perron--Frobenius) transfer operator, 
i.e.,
$$
\LL_a \varphi(x)=\sum_{T_a(y)=x} \frac{\varphi(y)}{|T_a'(y)|}\,. 
$$
The appropriate space of observables $V_\alpha$, $0<\alpha\le1$, 
for which $\LL_a$ has a spectral gap and which is convenient for our setting
 was introduced in \cite{keller} (see also 
\cite{wong}, \cite{hk}, and \cite{saussol} which treats the higher dimensional case). 
$V_\alpha$ is the space of functions of {\em generalised bounded variation}.

\begin{definition}[Banach space $V_\alpha$]
\label{d.generalizedbv}
For $\varphi\in L^1(m)$ and $\delta>0$, we define
$$
\osc(\varphi,\delta,x)=\esssup\varphi|_{(x-\delta,x+\delta)}-\essinf\varphi|_{(x-\delta,x+\delta)}\,,
$$
and, for $0<\alpha\le1$ and $A>0$, set
$$
|\varphi|_\alpha=\sup_{0<\delta\le A}\frac1{\delta^\alpha}
\int_0^1\osc(\varphi,\delta,x)dx\,.
$$
The space $V_\alpha$ consists of all $\varphi\in L^1(m)$ such that $|\varphi|_\alpha<\infty$. 
On $V_\alpha$ we define the norm
$$
\|\varphi\|_\alpha=|\varphi|_\alpha+\|\varphi\|_{L^1}\,.
$$
(Observe that the norm $\|\cdot\|_\alpha$ depends also on the constant $A$.)
\end{definition}

It follows immediately that $V_\alpha$ contains all $\alpha$-H\"older functions. 
Further, by \cite[Theorem~1.13]{keller} and \cite[Proposition~3.4]{saussol}, 
the space $V_\alpha$ together with 
the norm $\|.\|_\alpha$ is a Banach space and there exists a constant $C=C(\alpha)$ so that 
for all $\varphi_1,\varphi_2\in V_{\alpha}$ we have 
\begin{equation}
\label{eq.inftyalpha}
\|\varphi_1\|_{\infty}\le C\|\varphi_1\|_{\alpha}\,,
\end{equation}
and
\begin{equation}
\label{eq.algebra} 
\|\varphi_1\varphi_2\|_{\alpha}
\le C\|\varphi_1\|_{\alpha}\|\varphi_2\|_{\alpha}\,.
\end{equation}
Having introduced our main Banach space $V_\alpha$ we can now state our 
second condition. This condition is independent on the choice of the map $x_0$.

\begin{itemize}
\item[(II)]
$T_0$ is mixing and there exist constants $\epsilon>0$, $C\ge1$ and 
$0<\tilde\rho<1$ such that for all $\varphi\in V_\alpha$
\begin{equation}
\label{eq.uniformLY}
\|\LL_a^n\varphi\|_\alpha\le C\tilde\rho^n\|\varphi\|_{\alpha}+C\|\varphi\|_{L^1}\,.
\end{equation}
\end{itemize}

As already mentioned above the last condition is a bit more technical.  
It is used to guarantee that images by $x_n$ of ``most'' elements in $\PP_n$ are 
not too small (see Lemma~\ref{l.goodpartition}).
For an alternative condition see Remark~\ref{r.3p} below.

\begin{itemize}
\item[(III)]
There exists $\epsilon>0$ such that for all $\delta_0>0$ 
there exists a constant $C$ so that 
\begin{equation}
\label{eq.condII}
\sum_{\om\in\PP_n|[0,\epsilon]}\frac1{\|x_n|_\om'\|_\infty}\le Ce^{n^{\delta_0}}\,,\qquad\forall n\ge1\,.
\end{equation}
\end{itemize}

We can now state the main result of this paper. By Corollary~\ref{c.lil}, this 
result immediately implies the law of iterated logarithm. Recall the definition of $\sigma$ in 
\eqref{eq.sigma} (where the observable $\varphi$ is now in the space $V_\alpha$).

\begin{theorem}
\label{t.main}
Let $T_a:[0,1]\to[0,1]$, $a\in[0,1]$, be a piecewise expanding one-parameter family, 
satisfying properties~(i)-(iii) and condition~(II) for some $0<\alpha\le1$.
If for a $C^{1+\alpha}$ map $x_0:[0,1]\to[0,1]$ property~\eqref{eq.convenient} and
conditions~(I) and (III) are satisfied, then for all $\varphi\in V_\alpha$ such that 
$\sigma_0(\varphi)>0$ there exists $\epsilon>0$ so that 
the process $\xi_i:[0,\epsilon]\to\real$, $i\ge1$, defined by 
\begin{equation}
\label{eq.xigen}
\xi_i(a)=\frac1{\sigma_a(\varphi)}\Big(\varphi(x_i(a))-\int\varphi d\mu_a\Big)\,,
\end{equation}
satisfy the almost sure invariance principle for any error exponent $\gamma>2/5$.
\end{theorem}

We conclude this section with the proof of Lemma~\ref{l.transversality}.

\begin{proof}[Proof of Lemma~\ref{l.transversality}]
Recall that the boundary points $b_0(a),...,b_{p_0}(a)$ are Lipschitz continuous and let 
$L$ be their Lipschitz constant. By condition~(I), we can take $j_0\ge1$ be so large that
$|x_{j_0}'(0+)|\ge\sup_{a\in[0,1]}\sup_{x\in K(a)}|\partial_aT_a(x)|/(\la-1)+2L+1$.
Recall that by condition~(I) there exists a neighbourhood $V\subset[0,1]$, so that 
$x_i(a)\notin\{b_0(a),...,b_{p_0}(a)\}$, for all $a\in V\setminus0$ and $0\le i<j_0$.
Hence, by continuity, we find 
$\epsilon>0$ (where $[0,\epsilon)\subset V$) so that 
\begin{equation}
\label{eq.startcomp}
|x_{j_0}'(a)|\ge\frac{\sup_{a\in[0,1]}\sup_{x\in K(a)}|\partial_aT_a(x)|}{\la-1}+2L\,,\qquad\forall 
a\in(0,\epsilon)\,.
\end{equation}
Let $j\ge1$ and assume in the following formulas that, for the parameter values $a\in[0,\epsilon]$ under 
consideration, $x_j$ and $T_a^j$ are differentiable in $a$ and $x_0(a)$, respectively.
For $0\le k<j$ we have 
\begin{equation}
\label{eq.startcalc1}
x_j'(a)=(T_a^{j-k})'(x_k(a))x_k'(a)+\sum_{i=k+1}^j(T_a^{j-i})'(x_i(a))(\partial_aT_a)(x_{i-1}(a)),
\end{equation}
which implies
\begin{equation}
\label{eq.startcalc}
\frac{x_j'(a)}{(T_a^j)'(x_0(a))}
=\frac{1}{(T_a^k)'(x_0(a))}\left(x_k'(a)+\sum_{i=k+1}^j
\frac{(\partial_aT_a)(x_{i-1}(a))}{(T_a^{i-k})'(x_k(a))}\right).
\end{equation}
For $j>j_0$, choosing $k=0$ and $k=j_0$, respectively, we get the following upper and lower bounds: 
\begin{equation}
\label{eq.startcalc2}
\frac{2L}{|(T_a^{j_0})'(x_0(a))|}\le\left|\frac{x_j'(a)}{(T_a^j)'(x_0(a))}\right|\le
\sup_{a\in[0,\epsilon]}\left(|x_0'(a)|+\frac{\sup_{x\in K(a)}|\partial_aT_a(x)|}{\la-1}\right),
\end{equation}
where for the lower bound we used the assumption \eqref{eq.startcomp}. 
It is only left to show that for each $j\ge j_0$ 
the number of 
$a\in[0,\epsilon]$ which 
are not contained in any element $\om\in\P_j|[0,\epsilon]$ is finite. 
This is easily done by induction over $j$. 
Observe first that, by the assumption on $x_0$, $x_j(a)\in K(a)$ for all $j\ge0$ and $a\in[0,1]$. 
So the only case that prevents $a$ to be contained in any element of $\om\in\P_j|[0,\epsilon]$ is 
when $x_i(a)\in\{b_0(a),...b_{p_0}(a)\}$, for some $i<j$.
By the choice of $\epsilon$ above inequality \eqref{eq.startcomp}, 
 only $0$ and $\epsilon$ might not be contained in any element of 
 $\om\in\P_{j_0}|[0,\epsilon]$.
Assume that $j\ge j_0$ and consider the partition $\P_{j+1}|[0,\epsilon]$. 
From the lower bound in \eqref{eq.startcalc2}, we derive that 
$|x_j'(a)|\ge\la^{j-j_0}2L>L$ for all $a$ contained in an element of $\P_j|[0,\epsilon]$. 
Since the boundary points $b_k(a)$ are $\Lip(L)$, 
we have that $x_j(a)\in K(a)\setminus\{b_1(a),...,b_k(a)\}$ for all but finitely many $a\in[0,\epsilon]$.
Hence, by the induction assumption we conclude that 
the number of 
$a\in[0,\epsilon]$ which 
are not contained in any element $\om\in\P_{j+1}|[0,\epsilon]$ is finite. 
This concludes the proof of Lemma~\ref{l.transversality}.
\end{proof}


\section{Tent maps and other examples}
\label{s.examples}
In this section we give some examples of piecewise expanding one-parameter families 
to which Theorem~\ref{t.main} can be applied. 

We start with a trivial example which provides a good insight regarding the technical 
condition~(III). Let $T_0:[0,1]\to[0,1]$ be a mixing piecewise expanding map admitting a unique 
acip $\mu_0$ with support, say, $[0,1]$. Let $\varphi\in V_\alpha$ so that $\sigma_0(\varphi)>0$. 
We will deduce the well-known fact that the functions $\varphi(T_0^j(x))-\int\varphi d\mu_0$, 
$j\ge1$,
satisfy the ASIP (see, e.g., \cite{hk}) from Theorem~\ref{t.main}:
As the one parameter family we take the constant 
family $T_a\equiv T_0$, for all $a\in[0,1]$. The map $x_0$ is the identity, i.e., $x_0(a)=a$. 
Obviously the transversality condition~(I) is satisfied. The Lasota-Yorke inequality 
for the map $T_0$ which we need follows from \cite[Theorem~3.2]{keller}. 
In order 
to apply Theorem~\ref{t.main}, the remaining condition to verify is condition~(III). 
Let $h=d\mu_0/dm$ be the density of $\mu_0$. 
By \cite{keller78} and \cite{Ko}, 
there exists a constant $C$ so that 
$$
\frac1{C}\le h(x)\le C\,,\qquad\text{for a.e. }x\in[0,1]\,.
$$
Since $h$ is a fixed point of the transfer operator $\LL_0$, we derive that
\begin{equation}
\label{eq.fix}
\sum_{T_0^n(y)=x} \frac1{|(T_0^n)'(y)|}
\le C\sum_{T_0^n(y)=x} \frac{h(y)}{|(T_0^n)'(y)|}
= Ch(x)\le C^2\,,
\end{equation}
for a.e. $x$.
Recall that the elements $\PP_1(0)$ are of the form $(b_{i-1},b_i)$, for $1\le i\le p_0$, and 
observe that the set of boundary points $\{\partial T_0^n(\om)\mid\om\in\PP_n(0)\}$ consists of 
maximally $2np_0$ points. This implies that we can find points $x_1,...,x_k$, $k\le2np_0$,
which lie close to this set, so that 
\begin{equation}
\label{eq.fixx}
\sum_{\om\in\PP_n(0)}\frac1{\|(T_0^n)'|_\om\|_\infty}
\le C\sum_{i=1}^k\sum_{T_0^n(y)=x_i} \frac1{|(T_0^n)'(y)|}
\le2C^3p_0n\,.
\end{equation}
(In the first inequality we used also a  standard 
distortion estimate for piecewise expanding maps; see, e.g., \eqref{eq.distortion2} below.)
Since, by definition, $x_n'(a)=(T_0^n)'(a)$ and $\PP_n|[0,1]=\PP_n(0)$, this 
concludes the verification of condition~(III). 
Observe that in this trivial setting the right hand side of \eqref{eq.condII} is only increasing linearly 
in $n$.

We continue by studying some non-trivial examples, first the tent maps which is the main 
purpose of this paper and then $\beta$-transformations and 
Markov partition preserving families. In the end of this section, we give an application 
of our results in order to obtain almost sure typicality results similar to the ones in \cite{s2}.

\subsection{Tent maps}
Let $T_a:[0,1]\to[0,1]$, $a\in[0,1]$, be a one-parameter family of tent maps, i.e., there 
exist $1<\lambda\le\Lambda<\infty$ and $0<\alpha\le1$ so that, for each $a\in[0,1]$, the map
$T_a:[0,1]\to[0,1]$ is continuous and there exists a turning point 
$c_a\in(0,1)$ such that $T_a|_{[0,c_a]}$ and $T_a|_{[c_a,1]}$ are $C^{1+\alpha}$, 
$0<\alpha\le1$, (where 
the H\"older constants, see \eqref{eq.holder}, are uniform in $a$), 
$\lambda\le|T_a'(x)|\le\Lambda$, for all $x\neq c_a$, 
and $T_a(1)=T_a(0)=0$. Regarding the parameter dependency we assume that 
properties~(i) and (ii) in the beginning of Section~\ref{s.uniform} are satisfied. 
Recall the definition \eqref{eq.transversal} of a transversal family of tent maps $T_a$.

\begin{theorem}
\label{t.tenttransversal}
Assume that the family $T_a$ be is transversal at $T_0$.
Further, assume that $T_0$ is mixing and that the turning point $c_0$ is either not periodic or if 
$p$ is its period then 
\begin{equation}
\label{eq.period}
\lambda^{\alpha p}>2\,.
\end{equation}
If $\varphi\in V_\alpha$  
so that $\sigma_0(\varphi)>0$, then there exists $\epsilon>0$ 
such that for almost every $a\in[0,\epsilon]$ the 
turning point $c_a$ satisfies the LIL 
for the function $\varphi$ under the map $T_a$.
\end{theorem}

\begin{proof}
In order to prove Theorem~\ref{t.tenttransversal}, we will verify conditions~(I)--(III). Then 
we can apply Theorem~\ref{t.main} and Corollary~\ref{c.lil} which concludes the proof. 
(We have also to make sure that property~(iii) in Section~\ref{s.uniform} is satisfied. 
This will follow, as a by-product, from the second last paragraph in this proof.)

Regarding condition~(I), we define the map $x_0:[0,1]\to[0,1]$ 
as $x_0(a)=T_a^{j_0}(c_a)$, where $j_0\ge1$ is so large that \eqref{eq.transversal} 
holds for all $j\ge j_0$. Observe that, since $T_0$ is piecewise expanding 
and by \eqref{eq.transversal}, we find a constant $\delta>0$ 
so that $x_0(a)\notin\{0,c_a,1\}$, for all $a\in(0,\delta]$ (otherwise, in a neighbourhood of $a=0$, 
$c_a$ would be pre-periodic and hence 
$|x_j'(0)|$ would be bounded in $j$ contradicting the transversality \eqref{eq.transversal}). 
Hence, property \eqref{eq.convenient} 
is satisfied for $x_0$ on the interval $[0,\delta]$. As before, 
using once more \eqref{eq.transversal}, 
for each $j\ge1$, we find a neighbourhood $V\subset[0,\delta]$ of $0$ so that $x_j(a)\neq c_a$, 
for all $a\in V\setminus 0$ (otherwise $|x_j'(a)|$ would be bounded). 
We conclude that $x_0$ satisfies condition~(I) (and
we can assume that $x_0$
satisfies \eqref{eq.convenient} on the interval $[0,1]$).

We continue with the verification of condition~(II) 
which is a condition on the family and which does not involve the 
map $x_0$. The problem in verifying condition~(II) is to get uniform constants in the 
Lasota-Yorke inequality. If $c_0$ is not periodic let $p$ be so large so that 
also in this non-periodic case inequality~\eqref{eq.period} is satisfied. 
\cite[Theorem~3.2]{keller} and its proof shows that for all $\delta>0$ and 
all $a\in[0,1]$ we find a constant 
$C=C(\delta,a)$ and $A=A(\delta,a)>0$ 
(recall that the norm $\|\cdot\|_\alpha$ depends also on $A$) 
so that, setting $\rho=(2+\delta)/\lambda^{\alpha p}$, we have
\begin{equation}
\label{eq.lasota}
\|\LL_a^p\varphi\|_\alpha\le\rho\|\varphi\|_\alpha+C\|\varphi\|_{L^1}\,.
\end{equation}
By \eqref{eq.period}, we can fix $\delta>0$ so small that $\rho<1$.
Hence, if we show that in a neighbourhood of $0$ we can choose the constants $C$ and $A$ 
uniformly in $a$, 
then \eqref{eq.lasota} combined with the assumption that $T_0$ is mixing implies condition~(II).
In order to verify this uniformity of $C$ and $A$, we have to show that 
the constants $K$ and $A$ in \cite[Lemma~3.1]{keller} can chosen independently 
on $a$ in an neighbourhood of $0$.
Set $M=\delta\lambda^{-\alpha}(\delta/(16+2\delta))^{1-\alpha}$. 
By continuity we find an $\epsilon>0$ so that $T_a^i(c_a)\neq c_a$, for all $a\in[0,\epsilon]$ and 
all $1\le i\le p-1$. Hence, we find a constant $\kappa>0$ so that for all $a\in[0,\epsilon]$ 
the sizes of the intervals of monotonicity for $T_a^p:[0,1]\to[0,1]$ are larger than $\kappa$. 
This and the fact that $x\mapsto|(T_a^p)'(x)|^{-1}$ is $\alpha$-H\"older continuous on 
these monotonicity intervals imply that
there is an integer $k$ and a constant $0<\kappa'\le\kappa$ so that, 
for each $a\in[0,\epsilon]$, there is a refinement $\{I_1(a),I_2(a),...,I_k(a)\}$ of the partition of $[0,1]$ 
into monotonicity intervals of $T_a^p$ so that, for all $1\le j\le k$, we have $\kappa'\le|I_j(a)|\le2\kappa'$ and
$$
\sup_{\begin{subarray}{c}b_0<b_1<....<b_n\\ b_0,...,b_n\in I_j(a)\end{subarray}}
\sum_{i=1}^n\Big(\Big||(T_a^p)'(b_{i+1})|^{-1}-|(T_a^p)'(b_{i})|^{-1}\Big|^{1/\alpha}\Big)^\alpha
< M\,.
$$
By this choice of $\{I_1(a),...,I_k(a)\}$, we easily see that properties $(16)$ and $(17)$ in the proof of 
\cite[Lemma~3.1]{keller} are satisfied. Further, setting $A=\kappa'\delta/(16+2\delta)$ 
corresponds to (17) in \cite{keller}. 
The remaining part of the proof of \cite[Lemma~3.1]{keller} 
immediately shows then that the constant $K$ therein only depends on the constants 
$M$, $\delta$, and $\kappa'$ which are by construction independent on $a\in[0,\epsilon]$.

It is left to verify condition~(III). Let $h_a=d\mu_a/dm$ denote the density of 
the acip for $T_a$. 
We show first that there is a positive lower bound of $h_a$ on its support which is 
uniform in $a$ close to $0$, i.e., there exists a constant $H<\infty$ so that 
\begin{equation}
\label{eq.habelow}
\essinf_{x\in K(a)}h_a(x)\ge H^{-1}\,,\qquad\text{for all $a$ close to $0$}.
\end{equation}
We claim that there exist $\epsilon>0$ and an integer $N\ge1$ so that, for all $a\in[0,\epsilon]$, 
there is an interval $I\subset K(a)$ of length $1/N$ so that $\essinf_{x\in I}h_a(x)\ge1/2$. 
We show this claim by contradiction.
By condition~(II) (see \eqref{eq.ha} below), 
we find constants $\epsilon>0$ and $C$ so 
that, for all $a\in[0,\epsilon]$, we have the bound $\|h_a\|_\alpha\le C$. 
For $N\ge1$, 
divide the unit interval into $N$ disjoint intervals $I_1,...,I_N$ of length $1/N$. 
For $1\le\ell\le N$, let $M_\ell(a)$ and $m_\ell(a)$ 
denote the essential supremum and the essential infimum of $h_a$ on $I_\ell$, respectively. 
Since $1=\int_0^1h_a(x)dx\le\sum_{\ell=1}^NM_\ell(a)/N$, we get $\sum_{\ell=1}^NM_\ell(a)\ge1$. 
Now, if the claim was not true, we find $a\in[0,\epsilon]$ so that $m_\ell(a)\le1/2$, for all 
$1\le\ell\le N$. From this we deduce
\begin{align*}
1/2&=1-1/2\le\sum_{\ell=1}^N(M_\ell(a)-m_\ell(a))/N\le\int_0^1\osc(h,1/N,x)dx\\
&\le\|h_a\|_\alpha/N^\alpha\le C/N^\alpha\,.
\end{align*}
Since the right hand side tends to zero for $N\to\infty$ we get a contradiction. 
Henceforth, fix $\epsilon>0$ and 
$N\ge1$ so that the just proven claim holds and, for $a\in[0,\epsilon]$, 
let $I(a)$ be the interval of length $1/N$ so that $\essinf_{x\in I(a)}h_a\ge1/2$.
We turn to the proof of \eqref{eq.habelow}.
By the expansion of $T_a$, it follows that there exists an integer $0\le k_0\le\ln N/\ln\la$
such that $c_a\in T_a^{k_0}(I(a))$. 
Let $0<\epsilon'\le\epsilon$, be so that $T_a$ is mixing for all $a\in[0,\epsilon']$ 
(this is possible by condition~(II); see the beginning of the proof of Proposition~\ref{p.udc}). 
Note that $T_a$ mixing implies that the support $K(a)$ of the acip is equal to 
$[T^2_a(c_a),T_a(c_a)]$. From this we derive that property~(iii) in Section~\ref{s.uniform} 
is satisfied.
By \cite{wagner} and since $T_a$ is mixing, 
we have that $T_a:K(a)\to K(a)$ is exact, i.e., for each set $S\subset K(a)$
of positive Lebesgue measure it follows that $\lim_{j\to\infty}|K(a)\setminus T_a^j(S)|=0$.
Observe that, since $T_a$ is a tent map, 
if $J$ is an interval of length close to $K(a)(=[T^2_a(c_a),T_a(c_a)])$ 
then we have $T_a^2(K(a))=K(a)$.
Thus, exactness implies that there is an integer $k_1$ such that 
$T_a^{k_1}([c_a-1/2N,c_a])=T_a^k([c_a,c_a+1/2N])=K(a)$.
Since the image of an interval by $T_a^j$, $j\ge1$, 
changes continuously in $a$ 
we can choose the integer $k_1$ 
independently on $a\in[0,\epsilon]$. Hence,
we conclude that $T_a^{k_0+k_1}(I(a))=K(a)$, for all $a\in[0,\epsilon]$. 
Using the equality 
$$
h_a(x)=\sum_{T_a^{k_0+k_1}(x)=y}\frac{h_a(y)}{|(T_a^{k_1+k_2})'(y)|}\,,\qquad\text{for a.e. }x\,,
$$
the desired property \eqref{eq.habelow} follows.

Let $\epsilon>0$ be the constant in Lemma~\ref{l.transversality}.
It is shown in \cite[Section~6.3]{s2} that 
there exists $0<\epsilon'\le\epsilon$ so that 
without loss of generality (otherwise inverse the order) if 
$0\le a_1\le a_2\le\epsilon'$ then for all $\om_1\in\PP_n(a_1)$, $n\ge1$, there exists 
(exactly) one $\om_2\in\PP_n(a_2)$ so that $\om_1$ and $\om_2$ have the same 
combinatorics up to the iteration $n-1$. 
In order to apply the distortion estimate \eqref{eq.distortion2} below, we divide the 
interval $[0,\epsilon']$ into smaller intervals.
For $n\ge1$, let $\II_n$ be a 
partition of $[0,\epsilon']$ into 
intervals $I$ of length approximately equal to $\epsilon'/n^{1/\alpha}$. 
For $I\in\II$, let $a_I$ denote the right boundary point of $I$. 
By the proof of Lemma~\ref{l.transversality}, it immediately follows that each two disjoint 
elements in $\PP_n|[0,\epsilon']$ have different combinatorics up to $n-1$. 
Hence, for $I\in\II$, there exists an injective map from $\PP_n|I$ to $\PP_n(a_I)$ 
which maps each element in $\PP_n|I$ to the element in $\PP_n(a_I)$ with the 
same combinatorics up to $n-1$.
Using Lemma~\ref{l.transversality} 
and the distortion estimate 
\eqref{eq.distortion2} below, we derive 
$$
\sum_{\om_1\in\PP_n|I}\frac1{\|x_n'|_{\om_1}\|_\infty}
\le C\sum_{\om_2\in\PP(a_I)}\frac1{\|(T_{a_I}^n)'|_{\om_2}\|_\infty}\le C^2n\,,
$$
where the last inequality follows by \eqref{eq.habelow}, \eqref{eq.fix}, and \eqref{eq.fixx} 
(\eqref{eq.habelow} guarantees that the constant $C$ does not depend on $a$). 
Now, we can sum over the intervals in $\II_n$ which concludes the verification of 
condition~(III) (where 
the right hand side in \eqref{eq.condII} increases in this setting like $n^{1+1/\alpha}$).
\end{proof}

Instead of taking the turning points $c_a$ as the points of interest 
we can choose arbitrary points $x_0(a)\in[0,1]$, as long as the transversality 
condition~(I) is satisfied. 
However, in order to verify condition~(III), we will still assume that the family itself is
transversal at $T_0$. (It is quite likely that with some more work this assumption can be 
dropped.)

\begin{theorem}
\label{t.tent}
Assume that the family $T_a$ is transversal at $T_0$. 
Further, assume that $T_0$ is mixing and that the turning point $c_0$ is either not periodic or if 
$p$ is its period then \eqref{eq.period} is satisfied.  
Let $x_0:[0,1]\to[0,1]$ be a $C^{1+\alpha}$ map so that the transversality 
condition~(I) is satisfied. 
If $\varphi\in V_\alpha$
so that $\sigma_0(\varphi)>0$, then there exists $\epsilon>0$ 
such that for almost every $a\in[0,\epsilon]$ the point $x_0(a)$ satisfies the LIL 
for the function $\varphi$ under the map $T_a$.
\end{theorem}

\begin{proof}
Condition~(II) for the family $T_a$ is already verified in the proof of Theorem~\ref{t.tenttransversal}. 

Observe that in Theorem~\ref{t.tent} we do not assume that $x_0$ satisfies \eqref{eq.convenient}.
However, we can make the following reasoning. 
Observe that, for all $a\in[0,1]$, all points in $(0,1)$ are mapped after a finite number of iteration into 
$[T_a^2(c_a),T_a(c_a)]$.
As explained in the beginning of the 
proof of Proposition~\ref{p.udc} below, the fact that condition~(II) is satisfied gives a constant 
$0<\epsilon'\le\epsilon$ so that $T_a$ is mixing for all $a\in[0,\epsilon']$. 
Hence, $K(a)=[T_a^2(c_a),T_a(c_a)]$, for all $a\in[0,\epsilon']$. 
Since condition~(I) is satisfied, we find $0<\epsilon''\le\epsilon'$ and an iteration $k\ge0$ 
so that $x_k(a)\in[T_a^2(c_a),T_a(c_a)]\setminus\{0,c_a,1\}$, for all $a\in[0,\epsilon'']$.
Hence, renaming $x_k$ by $x_0$ (and considering the smaller interval $[0,\epsilon'']$),
without loss of generality, we can assume in the 
remaining part of this proof that $x_0$ satisfies \eqref{eq.convenient}. 

Regarding condition~(III) we note that property \eqref{eq.habelow} also holds in the setting 
of Theorem~\ref{t.tent}. Then we can follow word by word the last paragraph in 
the proof of Theorem~\ref{t.tenttransversal} which concludes the verification of 
condition~(III).
\end{proof}

\subsection{Generalised $\beta$-transformations and Markov partition preserving families}
First we consider a generalised form of $\beta$-transformations. 
Let $T:[0,\infty)\to[0,1]$ be piecewise $C^{1+\alpha}$, $0<\alpha\le1$, 
and $0=b_0<b_1<...$ be the associated partition, 
where $b_k\to\infty$ as $k\to\infty$. We assume that
$T$ is right continuous and $T(b_k)=0$, for each $k\ge0$.
Further, for each $a>1$, we have $\|T'(a\, \cdot)^{-1}\|_{L^\infty([0,1])}<1$ and 
$\|T'(a\, \cdot)\|_{L^\infty([0,1])}<\infty$. 
For $a_0>1$, we define the one-parameter family $T_a:[0,1]\to[0,1]$, $a\in[0,1]$, by 
$T_a(x)=T((a_0+a)x)$. 
It is shown in \cite[Lemma 5.1]{s2} that each $T_a$ admits a unique acip $\mu_a$ 
whose support $K(a)$ is an interval adjacent to $0$. Further, the length of K(a) is an 
increasing, piecewise 
constant function in $a$ where the discontinuities are isolated point.
Let $\lambda(a)=\essinf_{x\in[0,1]}|T_a'(x)|$. 
Regarding the verification of condition~(II), we make sure 
that a similar condition as in \eqref{eq.period} is satisfied: 
We assume that $b_j/a_0\neq1$, for all $j\ge0$, and there exists $p\ge1$ such that 
\begin{equation}
\label{eq.periodbeta}
\lambda(a_0)^{\alpha p}>2\,,\ \  \text{and }\ T^i(b_j-)\neq b_k/a_0\,.
\end{equation}
for all $1\le i\le p-1$ and $k\ge 1$. 
Furthermore, we assume that $|K(a)|$ is constant in a neighbourhood of $a=0$.

\begin{theorem}
\label{t.beta}
Let $x_0:[0,1]\to[0,1]$ be a $C^{1+\alpha}$ map satisfying condition (I).
If $\varphi\in V_\alpha$  
so that $\sigma_0(\varphi)>0$, then there exists $\epsilon>0$ 
such that for almost every $a\in[0,\epsilon]$ the 
turning point $x_0(a)$ satisfies the LIL 
for the function $\varphi$ under the map $T_a$.
\end{theorem}

We continue with one-parameter families preserving a Markov structure. 
Assume that we have a one-parameter family $T_a:[0,1]\to[0,1]$, $a\in[0,1]$, as described 
in the beginning of 
Section~\ref{s.uniform} with a partition $0\equiv b_0(a)<b_1(a)<...<b_{p_0}(a)\equiv1$ and 
satisfying properties (i)-(iii). We require additionally that the family 
$T_a$ fulfils the following Markov property. Set $B_k(a)=(b_{k-1}(a),b_k(a))$, $1\le k\le p_0$.

\begin{itemize}
\item[(M)]
For each $1\le k\le p_0$ the image $T_a(B_k(a))$, $a\in[0,1]$, is a union 
of monotonicity intervals $B_\ell(a)$, $1\le \ell\le p_0$ 
(modulo a finite number of points).
\end{itemize}

\begin{theorem}
\label{t.markov}
Let $T_a$ be a family satisfying the Markov property (M) and 
let $x_0:[0,1]\to[0,1]$ be a $C^{1+\alpha}$ map satisfying condition (I). 
If $\varphi\in V_\alpha$  
so that $\sigma_0(\varphi)>0$, then there exists $\epsilon>0$ 
such that for almost every $a\in[0,\epsilon]$ the 
turning point $x_0(a)$ satisfies the LIL 
for the function $\varphi$ under the map $T_a$.
\end{theorem}

\begin{proof}[Proof of Theorems~\ref{t.beta} and \ref{t.markov}]
Due to the Markov structure, the proof of Theorem~\ref{t.markov} is much easier 
than the proofs of Theorem~\ref{t.tenttransversal}, \ref{t.tent}, and \ref{t.beta}. We 
leave it as an exercise to the reader. The proof of Theorem~\ref{t.beta} is very similar to 
the proof of Theorem~\ref{t.tenttransversal}. 
 Regarding property \eqref{eq.convenient} we can argue 
as in the proof of Theorem~\ref{t.tent}. The fact that $T_0$ is mixing is shown 
in the last paragraph in \cite[Section~5.2]{s2}. Property~\eqref{eq.periodbeta}, ensures that 
we can go word by word along the verification of condition~(II) in the proof of 
Theorem~\ref{t.tenttransversal}. Knowing that condition~(II) is satisfied ensures that 
$\sigma_a(\varphi)>0$ in an neighbourhood of $0$ (see Lemma~\ref{l.varregularity} below). 
It remains to verify condition~(III). Observe that, by the construction of the family $T_a$, 
if $0\le a_1\le a_2\le1$ then for all $\om_1\in\PP_n(a_1)$, $n\ge1$, there exists  
$\om_2\in\PP_n(a_2)$ so that $\om_1$ and $\om_2$ have the same 
combinatorics up to the iteration $n-1$. Hence, if we show that the densities are uniformly 
bounded below on their support (see \eqref{eq.habelow}), we can follow 
the last paragraph in the proof of Theorem~\ref{t.tenttransversal} which concludes the 
verification of condition~(III). The only obstacle in showing \eqref{eq.habelow} might be the case 
when $K(0)$ is smaller than $K(a)$ but this case is excluded by our assumption 
on the family $T_a$. The proof of \eqref{eq.habelow} 
in a neighbourhood of $a=0$ is done in detail in \cite[inequality~(30)]{s2}. 
\end{proof}

\subsection{Almost sure typicality}
Let $T_a:[0,1]\to[0,1]$, $a\in[0,1]$, be a one-parameter family of 
piecewise expanding maps as described in Section~\ref{s.uniform} and satisfying 
properties~(i)-(iii) therein. Let $x_0:[0,1]\to[0,1]$ be a $C^{1+\alpha}$ map 
satisfying \eqref{eq.convenient}. As above let $h_a$ denote the density of $\mu_a$.
As a corollary of Theorem~\ref{t.main} we get the following typicality result. 
Recall the definition of typical in \eqref{eq.typical}.

\begin{theorem}
\label{t.typical}
If conditions~(I)-(III) are satisfied and if there exists $\epsilon>0$ and a constant $C$ 
so that 
\begin{equation}
\label{eq.densbelow}
\essinf_{x\in K(a)}h_a(x)\ge C^{-1}\,,\qquad\forall a\in[0,\epsilon]\,,
\end{equation}
then there exists $0<\epsilon'\le\epsilon$ so that 
$x_0(a)$ is typical for $\mu_a$ for a.e. $a\in[0,\epsilon']$. 
\end{theorem}

\begin{proof}
For $\kappa>0$ small, let 
$$
\BB=\{(q-r,q+r)\cap[0,1]\mid q\in\rational,\, r\in\rational\cap[0,\kappa]\}\,.
$$
Observe that in order to prove 
Theorem~\ref{t.typical}, it is sufficient to 
show that there exists an $\epsilon'>0$ so that, for each $B\in\BB$, 
$x_0(a)$ satisfies the LIL for  $\chi_B$ under the map $T_a$, for a.e. $a\in[0,\epsilon']$. 
From the proof of Theorem~\ref{t.main}, we see that the constant $\epsilon$ in 
the assertion of Theorem~\ref{t.main} does only depend on the constant $\epsilon'$ in 
Proposition~\ref{p.udc} and the length of the interval of parameters $a$ on which 
$\sigma_a(\varphi)>0$. Since $\epsilon'$ in Proposition~\ref{p.udc} 
does only depend on the family $T_a$ and not on the 
observable $\varphi$, it is enough to show that there exists $\delta>0$ so that 
 $\sigma_a(\chi_B)>0$, for all $B\in\BB$ and all $a\in[0,\delta]$. 
By Proposition~\ref{p.udc} and \eqref{eq.ha} below, and \eqref{eq.inftyalpha}, we find $\delta>0$, 
$C$, and $0<\rho<1$ so that, for all $a\in[0,\delta]$,
we have $\|h_a\|_\infty\le C\|h_a\|_\alpha/2\le C$
and, for all $B\in\BB$ and $a\in[0,\delta]$, we have
\begin{align*}
\Big|\int\chi_B\chi_B\circ T_a^nd\mu_a-\Big(\int\chi_Bd\mu_a\Big)^2\Big|
&\le C\|\chi_Bh_a\|_\alpha\|\chi_B\|_{L^1}\rho^n\\
&\le C^2|B|\rho^n\,,\qquad\forall n\ge1\,,
\end{align*}
where in the last inequality we used also \eqref{eq.algebra}. Altogether, for $a\in[0,\delta]$, 
we derive
\begin{align*}
\sigma_a(\chi_B)^2&=\int\chi_Bd\mu_a-\Big(\int\chi_Bd\mu_a\Big)^2
+2\sum_{n\ge1}\int\chi_B\chi_B\circ T_a^nd\mu_a-\Big(\int\chi_Bd\mu_a\Big)^2\\
&\ge C^{-1}|B|-2NC^2|B|^2-2\sum_{n\ge N}C^2|B|\rho^n\,,\qquad\forall N\ge1\,.
\end{align*}
Now, by taking $\kappa>0$ in the definition of $\BB$ sufficiently small, we can choose 
$N$ so that $\sigma_a(\chi_B)^2\ge |B|/2C$, for all $B\in\BB$ and all $a\in[0,\epsilon]$. 
This concludes the proof of Theorem~\ref{t.typical}.
\end{proof}

\begin{remark}
The question of typicality of a point $x_0(a)$ for almost every parameter $a$ in 
a general setting, was already studied 
in \cite{s2} (see also \cite{bruin}, \cite{fp}, and \cite{schmeling} for more specific cases). 
Theorem~\ref{t.typical} provides some alternative conditions.
The method in \cite{s2} is inspired by a technique developed in \cite{BC1} (see also 
\cite{bjs} for another application of this technique).
 This method is very different from the one used in the present paper. 
\end{remark}


\section{Preliminaries regarding the proof of Theorem~\ref{t.main}}
\label{s.prel}
In this section, we fix an $\epsilon>0$ 
which is at least so small as in Lemma~\ref{l.transversality} and 
conditions~(II) and (III). When the meaning is clear, 
we will write $\PP_j$ instead of $\PP_j|[0,\epsilon]$.

We start with an elementary but important 
statement about the size of exceptionally small partition elements. 
Since we are far away from having Markov partitions, the 
image $x_j(\om)$ of a partition element $\om$ in $\PP_j$ might be very small 
(despite the expansion of the map $x_j:\om\to[0,1]$). If this image is too small it contains not 
sufficient information in order to use it in our analysis. From condition~(III) we can derive a good 
control of the total size of partition elements having too small images for our purpose. 

\begin{lemma}
\label{l.goodpartition}
Assume that condition~(III) is satisfied.
Let $d_j>0$, $j\ge1$, be a sequence decaying at least stretched exponentially fast, i.e., 
there exists $\delta>0$ so that 
\begin{equation}
\label{eq.stretched}
\lim_{j\to\infty}d_j/e^{-j^\delta}<\infty\,.
\end{equation}
There exists a constant $C$ such that, for all $j\ge1$, 
the size of the exceptional set $E_j:=\{\om\in\PP_j\mid |x_j(\om)|\le d_j\}\subset\PP_j$,
has the upper bound
$$
\Big|\bigcup_{\om\in E_j}\om\Big|\le Cd_j^{1/2}\,.
$$
\end{lemma}

\begin{proof}
Take $\delta_0$ in condition~(III) strictly less than a $\delta$ satisfying \eqref{eq.stretched}. 
By the distortion estimate \eqref{eq.distortion1} below,  
for $\om\in\PP_j$ such that $|x_j(\om)|\le d_j$,
we have $|\om|\le C\frac{d_j}{\|x_j|_\om'\|_\infty}$.
We conclude that 
$$
\Big|\bigcup_{\om\in E_j}\om\Big|\le Cd_j\sum_{\om\in\PP_j}\frac1{\|x_j|_\om'\|_\infty}
\le C^2d_je^{j^{\delta_0}}\le C^3d_j^{1/2}\,.
$$
\end{proof}

\begin{remark}
\label{r.3p}
Lemma~\ref{l.goodpartition} is the only place where we need condition~(III). As an alternative 
condition to (III) it would be sufficient to require the following:
\begin{itemize}
\item[(III)']
For each $\delta>0$ there are constants $C$ and $\beta>0$ so that 
$$
|\{\om\in\PP_j\mid |x_j(\om)|\le e^{-j^\delta}\}|\le Ce^{-j^\beta}\,.
$$
\end{itemize}
We preferred to put the slightly stronger condition~(III) in Section~\ref{s.uniform} 
since it is the condition which we actually verify in the examples considered in 
Section~\ref{s.examples}.
\end{remark}

Since the sequence of maps $x_j$ is not the iteration of a fixed dynamical system
admitting an invariant measure, in order to gain information about this sequence we 
have to switch locally from $x_j$ to $T_{a_0}^j$ for some fixed parameter value $a_0$. 
After having switched we can profit from the abundant existing results for such a fixed 
mixing piecewise expanding map $T_{a_0}$. Very frequently 
we will use the exponential 
decay of correlations of $T_{a_0}$. Since we can only switch locally, we need that the 
constants in the decay of correlation for different $T_a$ in the family are uniform.

\begin{proposition}[Uniform decay of correlations]
\label{p.udc}
Assume that the family $T_a$ satisfies condition~(II).
Then, the family $T_a$ has uniform exponential decay of correlations for $a$ close to $0$, i.e., 
there exist constants $0<\epsilon'\le\epsilon$, 
$C\ge1$, and $0<\rho<1$ such that for all $a\in[0,\epsilon']$, 
for all functions $\varphi\in V_\alpha$, and all $\psi\in L^1$ we have 
$$
\Big|\int_0^1\varphi\psi\circ T_a^ndm-\int_0^1\varphi dm\int_0^1\psi d\mu_a\Big|
\le C \|\varphi\|_{\alpha}\|\psi\|_{L^1}\rho^n\,,\quad\forall n\ge1\,.
$$
\end{proposition}

\begin{proof}
The proof is a direct application of the perturbation results of Keller and Liverani \cite{kellerliverani} using 
the estimates in Keller \cite{keller} and Saussol \cite{saussol}. Observe that $\|\LL_a^n\varphi\|_{L^1}=\|\varphi\|_{L^1}$. 
At the end of this proof we will show that for all $\varphi\in V_{\alpha}$
\begin{equation}
\label{eq.strongtoweak}
\|(\LL_a-\LL_0)\varphi\|_{L^1}/\|\varphi\|_{\alpha}=O(|a|^\alpha).
\end{equation}
Combined with \eqref{eq.uniformLY} in condition~(II) and since $T_0$ is mixing, 
by \cite{kellerliverani}, we find $0<\epsilon'\le\epsilon$ so that, for all $a\in[0,\epsilon']$,
$T_a$ is mixing and $\LL_a$ can be written as $\PPP_a+\QQ_a$ 
where $\PPP_a\varphi=h_a\int\varphi dm$ is a one-dimensional projection and 
where there are constants $C\ge1$ and $0<\rho<1$ (both independent on $a$) so that 
$\|\QQ_a^n\varphi\|_{\alpha}\le C\rho^n\|\varphi\|_{\alpha}$, for all $n\ge1$. 
Furthermore, for the later use we note that by \cite{kellerliverani} we get a constant 
$\kappa>0$ such that 
\begin{equation}
\label{eq.densitydiff}
\|h_0-h_a\|_{L^1}= O(|a|^\kappa)\,,\qquad\forall a\in[0,\epsilon']\,.
\end{equation}
For $\varphi\in V_{\alpha}$ and $\psi\in L^1$, we get
\begin{align*}
\int\varphi\psi\circ T_a^ndm=\int[(\PPP_a+\QQ_a^n)\varphi]\psi dm
=\int\varphi dm\int\psi d\mu_a+\int\psi\QQ_a^n\varphi dm\,.
\end{align*}
Hence, using \eqref{eq.inftyalpha}, we derive 
\begin{align*}
\Big|\int\varphi\psi\circ T_a^ndm-\int\varphi dm\int\psi d\mu_a\Big|
\le C^2\|\varphi\|_{\alpha}\|\psi\|_{L^1}\rho^n\,.
\end{align*}

It remains to show \eqref{eq.strongtoweak}. Recall the notation $b_0,...,b_{p_0}$ for the 
partition points (the $b_i$'s depend on $a$ and are Lipschitz in $a$, say with Lipschitz 
constant $L$). 
Observe that $\|(\LL_a-\LL_0)\varphi\|_{L^1}$ is bounded above by
$$
\sum_{i=0}^{p_0-1}\int\left|\frac{\varphi(T_a|_{[b_i,b_{i+1}]}^{-1}(x))}
{|T_a'(T_a|_{[b_i,b_{i+1}]}^{-1}(x))|}\chi_{T_a([b_i,b_{i+1}])}(x)-\frac{\varphi(T_0|_{[b_i,b_{i+1}]}^{-1}(x))}
{|T_0'(T_0|_{[b_i,b_{i+1}]}^{-1}(x))|}\chi_{T_0([b_i,b_{i+1}])}(x)\right|\,.
$$
Let $J_i$ be the interval $T_0([b_i,b_{i+1}])$ from which we 
subtract at each boundary point an interval of length $\Lambda L|a|$. 
Since the partition points $b_i$ are Lipschitz in $a$, it follows that if 
$y\in (T_0|_{[b_i,b_{i+1}]})^{-1}(J_i)$ then $y\in(b_i(a'),b_{i+1}(a'))$, for all 
$a'\in[0,a]$. Furthermore, we have that $(T_a|_{[b_i,b_{i+1}]})^{-1}(J_i)\subset(b_i(0),b_{i+1}(0))$. 
Recall property~(ii) in the beginning of Section~\ref{s.uniform}, in particular, recall that 
$a'\mapsto T_{a'}'(y)$ is $\alpha$-H\"older. 
Restricting the integral above to the interval $J_i$, we apply the triangle inequality and 
we split the integral into two 
integrals where the first one is (recall \eqref{eq.inftyalpha})
\begin{multline*}
\int_{J_i}\left|\varphi(T_a|_{[b_i,b_{i+1}]}^{-1}(x))
\left(\frac1{|T_a'(T_a|_{[b_i,b_{i+1}]}^{-1}(x))|}-\frac1{|T_0'(T_0|_{[b_i,b_{i+1}]}^{-1}(x))|}\right)\right|dx\\
\le C\|\varphi\|_{L^\infty}|a|^\alpha\le C^2\|\varphi\|_{\alpha}|a|^\alpha\,,
\end{multline*}
and the second one is
\begin{multline*}
\int_{J_i}\frac1{|T_0'(T_0|_{[b_i,b_{i+1}]}^{-1}(x))|}
\left|\varphi(T_a|_{[b_i,b_{i+1}]}^{-1}(x))-\varphi(T_0|_{[b_i,b_{i+1}]}^{-1}(x))\right|dx\\
\le C\int_{J_i}\osc(\varphi,C|a|,y)dy\le C^2|a|^{\alpha}|\varphi|_\alpha
\le C^2|a|^\alpha\|\varphi\|_\alpha\,,
\end{multline*}
where we used the first inequality in \eqref{eq.forudc} below (therein set $x_{i+1}^1=x_{i+1}^2=x$).
In order to derive \eqref{eq.strongtoweak}, it remains only to consider the integrals over 
$T_a([b_i,b_{i+1}])\setminus J_i$ and $T_0([b_i,b_{i+1}])\setminus J_i$, respectively. 
However, one easily sees that the 
measures of these sets are bounded by a constant times $|a|$. Using once more 
\eqref{eq.inftyalpha}, this concludes the proof.
\end{proof}

The next lemma is a collection of various distortion estimates. Recall the notations of the 
partitions in Section~\ref{ss.partitions}. In particular, recall that $\PP_j(a)$ is 
the partition in the phase space, while $\PP_j(=\PP_j|[0,\epsilon])$ denotes the partition 
in the parameter space.

\begin{lemma}[Distortion]
\label{l.distortion}
There exists a constant $C$ such that the following holds.

For $a_1,a_2\in[0,\epsilon]$ and $k\ge1$, if $x\in[0,1]$, has the same combinatorics under $T_{a_1}$
and $T_{a_2}$ up to the $(k-1)$-th iteration, then 
\begin{equation}
\label{eq.xaa}
|T_{a_1}^k(x)-T_{a_2}^k(x)|\le C\Lambda^k|a_1-a_2|\,.
\end{equation}

Let $\ome\in\PP_k$. If $\om\subset\ome$ is an interval, then 
\begin{equation}
\label{eq.distortion1}
\left|\frac{x_k'(a_1)}{x_k'(a_2)}\right|\le\left(1+C|x_k(\om)|^\alpha\right)\,,\qquad\forall a_1,a_2\in\om\,.
\end{equation}

Let $k\ge1$ and $a_1,a_2\in[0,\epsilon]$ so that $|a_1-a_2|\le1/k^{1/\alpha}$. 
If $\om_1\in\PP_k(a_1)$ and $\om_2\in\PP_k(a_2)$ have the same 
combinatorics up to the $(k-1)$-th iteration then 
\begin{equation}
\label{eq.distortion2}
\left|\frac{(T_{a_1}^k)'(x_1)}{(T_{a_2}^k)'(x_2)}\right|\le C\,,\qquad\forall x_1\in\om_1\ \text{and }
x_2\in\om_2\,.
\end{equation}

Let $1\le k\le \ell$. 
For $\om\in\PP_\ell$ and $a\in\om$, we have
\begin{equation}
\label{eq.distortion3}
C^{-1}\le\left|\frac{x_\ell'(a)/x_k'(a)}{(T_a^{\ell-k})'(x_k(a))}\right|\le C\,.
\end{equation}
\end{lemma}

\begin{proof}
Property~\eqref{eq.distortion3} follows immediately from Lemma~\ref{l.transversality}.

We next show property~\eqref{eq.xaa}.
Set $x^1_i=T_{a_1}^i(x)$ and $x^2_i=T_{a_2}^i(x)$, $0\le i\le k-1$. We assume that the constant 
$C$ in the assertion of Lemma~\ref{l.distortion} satisfies $C\gg\delta_0^{-1}$ where 
$\delta_0$ is the constant in property~(i) in Section~\ref{s.uniform}. By this choice, 
regarding the proof of \eqref{eq.xaa} the only non-trivial situation is when $|a_1-a_2|\ll\delta_0$. 
Recall that the partition points $b_0(a)<...<b_{p_0}(a)$ are 
Lipschitz, say with constant $L$. 
Let $0\le i\le k-1$ and take $\ell=\ell(i)$ so that $x_i^1\in(b_{\ell-1}(a_1),b_\ell(a_1))$. 
Since $|a_1-a_2|\ll\delta_0$, we find $y\in(0,1)$ so that  
$|x_i^1-y|<L|a_1-a_2|/2$ and $y\in(b_{\ell-1}(a),b_\ell(a))$, for all $a\in[a_1,a_2]$. 
By property~(ii) in Section~\ref{s.uniform}, it follows then that 
$|T_{a_1}(y)-T_{a_2}(y)|\le C|a_1-a_2|$. Hence, we derive
\begin{align}
\label{eq.aaxx}
\nonumber
|T_{a_1}(x_i^1)-T_{a_2}(x_i^2)|
&\le|T_{a_1}(x_i^1)-T_{a_1}(y)|+|T_{a_1}(y)-T_{a_2}(y)|+|T_{a_2}(y)-T_{a_2}(x_i^2)|\\
\nonumber
&\le\Lambda|x_i^1-y|+C|a_1-a_2|+\Lambda|y-x_i^2|\\
&\le\Lambda(L+C)|a_1-a_2|+\Lambda|x_i^1-x_i^2|\,.
\end{align}
This estimate immediately implies \eqref{eq.xaa}.

Regarding property~\eqref{eq.distortion1} observe first that by \eqref{eq.startcalc} 
(when $k=0$ therein) we get
\begin{equation}
\label{eq.xkTk}
\frac{x_k'(a)}{(T_a^k)'(x_0(a))}
=x_0'(a)+\sum_{j=1}^k\frac{(\partial_aT_a)(x_{j-1}(a))}{(T_a^j)'(x_0(a))}\,.
\end{equation}

As in proving \eqref{eq.xaa}, we can assume that $|\om|\ll\delta_0$ (otherwise we can 
compensate by possibly increasing the constant $C$). 
We proceed similarly as in deriving \eqref{eq.aaxx}. 
Let $0\le i\le k-1$ and take $\ell=\ell(i)$ so that $x_i(a_1)\in(b_{\ell-1}(a_1),b_\ell(a_1))$. 
Since $|a_1-a_2|\ll\delta_0$, we find $y\in(0,1)$ so that  
$|x_i(a_1)-y|<L|a_1-a_2|/2$ and $y\in(b_{\ell-1}(a),b_\ell(a))$, for all $a\in[a_1,a_2]$. 
By property~(ii) in Section~\ref{s.uniform}, it follows that 
$|T_{a_1}'(y)-T_{a_2}'(y)|\le C|a_1-a_2|^\alpha$. 
Hence, by a similar calculation as in \eqref{eq.aaxx}, we get 
\begin{equation}
\label{eq.30}
|T_{a_1}'(x_i(a_1))-T_{a_2}'(x_i(a_2))|\le C|a_1-a_2|^\alpha + C|x_i(\om)|^\alpha
\le C^2|x_i(\om)|^\alpha\,.
\end{equation}
Thus, 
\begin{equation}
\label{eq.standarddist}
\left|\frac{(T_{a_1}^j)'(x_0(a_1))}{(T_{a_2}^j)'(x_0(a_2))}\right|
=\prod_{i=0}^{j-1}\left|\frac{T_{a_1}'(x_i(a_1))}{T_{a_2}'(x_i(a_2))}\right|
\le1+C\sum_{i=0}^{j-1}|x_i(\om)|^\alpha\,,
\end{equation}
from which follows that 
$$
\left|\frac1{(T_{a_1}^j)'(x_0(a_1))}-\frac1{(T_{a_2}^j)'(x_0(a_2))}\right|
\le C\frac{|x_{j-1}(\om)|^\alpha}{|(T_{a_1}^j)'(x_0(a_1))|}\,.
$$
Recall that, by property~(ii) in Section~\ref{s.uniform}, $a\mapsto\partial_aT_a(x)$ 
and $x\mapsto\partial_aT_a(x)$ are $\alpha$-H\"older continuous.
Hence, using a ``help'' point $y$ as above, we get 
$$
\left|(\partial_aT_a)|_{a=a_1}(x_{j-1}(a_1))-(\partial_aT_a)|_{a=a_2}(x_{j-1}(a_2))\right|
\le C|x_{j-1}(\om)|^\alpha\,.
$$
Combined with the $\alpha$-H\"older continuity of $x_0'$, by comparing each term 
on the right hand side of \eqref{eq.xkTk} for $a=a_1$ and $a=a_2$, it follows
\begin{align*}
\left|\frac{x_k'(a_1)}{(T_{a_1}^k)'(x_0(a_1))}\right|
&\le\left|\frac{x_k'(a_2)}{(T_{a_2}^k)'(x_0(a_2))}\right|
+C|a_1-a_2|^\alpha+C\sum_{j=1}^k\lambda^{-j}|x_{j-1}(\om)|^\alpha\\
&\le\left|\frac{x_k'(a_2)}{(T_{a_2}^k)'(x_0(a_2))}\right|+C^2|x_k(\om)|^\alpha\,.
\end{align*}
Altogether, we have
\begin{align*}
\left|\frac{x_k'(a_1)}{x_k'(a_2)}\right|\le\left(1+C|x_k(\om)|^\alpha\right)
\left|\frac{x_k'(a_1)/(T_{a_1}^k)'(x_0(a_1))}{x_k'(a_2)/(T_{a_2}^k)'(x_0(a_2))}\right|
\le1+C^4|x_k(\om)|^\alpha\,,
\end{align*}
where in the last inequality we use the fact that 
$|x_k'(a_2)/(T_{a_2}^k)'(x_0(a_2))|\ge C^{-1}$, by Lemma~\ref{l.transversality}. 

It is left to prove the distortion estimate~\eqref{eq.distortion2}.
Choose two points $x_0^1\in \om_1$ 
and $x_0^2\in\om_2$ and, for $1\le i\le k$, let 
$x_i^1=T_{a_1}^i(x_0^1)$ and $x_i^2=T_{a_2}^i(x_0^2)$.
We claim that there is a constant $C$ so that
\begin{equation}
\label{eq.distance}
|x_i^1-x_i^2|\le C\frac1{k^{1/\alpha}}+\frac1{\lambda^{k-i}}\,,\qquad\forall 0\le i\le k\,.
\end{equation}
In order to show ~\eqref{eq.distance}, we proceed similarly as in showing \eqref{eq.aaxx}. 
Let $0\le i\le k-1$ and take $\ell=\ell(i)$ so that $x_i^1\in(b_{\ell-1}(a_1),b_\ell(a_1))$. 
By possible increasing the constant $C$ in the assertion of Lemma~\ref{l.distortion} 
we can assume that $|a_1-a_2|\ll\delta_0$ and we find $y\in(0,1)$ so that  
$|x_i^1-y|<L|a_1-a_2|/2$ and $y\in(b_{\ell-1}(a),b_\ell(a))$, for all $a\in[a_1,a_2]$. 
Since $|y-x_i^2|\le\la^{-1}|T_{a_2}(y)-T_{a_2}(x_i^2)|$, 
we obtain
$$
|x_i^1-x_i^2|\le L|a_1-a_2|/2
+\frac1\la|T_{a_2}(y)-T_{a_1}(y)|+\frac1\la|T_{a_1}(y)-x_{i+1}^2|\,.
$$
As in \eqref{eq.aaxx}, we have
$|T_{a_2}(y)-T_{a_1}(y)|\le C|a_2-a_1|$, and note that
$$
|T_{a_1}(y)-x_{i+1}^2|\le|T_{a_1}(y)-x_{i+1}^1|+|x_{i+1}^1-x_{i+1}^2|
\le \Lambda L|a_1-a_2|/2+|x_{i+1}^1-x_{i+1}^2|\,.
$$
Altogether, recalling that $|a_1-a_2|\le1/k^{1/\alpha}$, we find a constant $C$ so that
\begin{equation}
\label{eq.forudc}
|x_i^1-x_i^2|\le C|a_1-a_2|+|x_{i+1}^1-x_{i+1}^2|/\lambda
\le C/k^{1/\alpha}+|x_{i+1}^1-x_{i+1}^2|/\lambda\,.
\end{equation}
From this estimate we easily deduce \eqref{eq.distance}.

By \eqref{eq.30} and \eqref{eq.distance}, for all 
$0\le i\le k-1$,
we obtain
$$
|T_{a_1}'(x_i^1)-T_{a_2}'(x_i^2)|\le C|a_1-a_2|^\alpha
+C|x_i^1-x_i^2|^\alpha\le C^2/k+C^2/\lambda^{\alpha(k-i)}\,,
$$
which implies
\begin{equation*}
\left|\frac{(T_{a_1}^k)(x_0^1)}{(T_{a_2}^k)'(x_0^2)}\right|\le
\prod_{i=0}^{k-1}\frac{|T_{a_1}'(x_i^1)|}{|T_{a_2}'(x_i^2)|}
\le\prod_{i=0}^{k-1}\frac{|T_{a_2}'(x_i^2)|+C^2/k+C^2/\lambda^{\alpha(k-i)}}{|T_{a_2}'(x_i^2)|}\,.
\end{equation*}
Since the right hand side is bounded by a constant independent on $k$, $a_1$ and $a_2$, 
this concludes the proof of \eqref{eq.distortion2}.
\end{proof}

Recall the definition of $\sigma_a(\varphi)$ in \eqref{eq.sigma} (where $\varphi$ here is 
in the space $V_\alpha$).
In order to ensure that the functions $\xi_j(a)$, $j\ge1$, defined in \eqref{eq.xigen} 
depend nicely on $a$, we have to 
investigate the $a$-dependence of $\sigma_a$.

\begin{lemma}[Regularity of $a\mapsto\sigma_a$]
\label{l.varregularity}
Assume that the family $T_a$ satisfies condition~(II).
Let $\epsilon'>0$ be the constant in Proposition~\ref{p.udc}.
For each $\varphi\in V_\alpha$ there exist constants $C$ and $\kappa>0$ such that 
\begin{equation}
\label{eq.varregularity}
|\sigma_a(\varphi)-\sigma_{a'}(\varphi)|\le C|a-a'|^\kappa\,,\quad\forall\ a,a'\in[0,\epsilon']\,.
\end{equation}
\end{lemma}

\begin{proof}
For simplicity we assume that $a'=0$ and $\int\varphi d\mu_0=0$.
The general case is proven similarly (cf. the last paragraph 
in this proof). 
For a constant $\kappa'>0$ to be determined later in the proof, 
let $k_0=k_0(a,0)$ be minimal such that, 
for $\bar a=0$ and $\bar a=a$, we have
$$
\Big|2\sum_{k> k_0}\int\Big(\varphi-\int\varphi d\mu_{\bar a}\Big)
\Big(\varphi-\int\varphi d\mu_{\bar a}\Big)\circ T_{\bar a}^kd\mu_{\bar a}\Big|\le a^{\kappa'}\,.
$$
By Proposition~\ref{p.udc}, the absolute value of the integral in the sum is bounded by 
a constant (independent on $a$) times $\rho^k$ which implies that 
\begin{equation}
\label{eq.k0}
k_0\le \kappa''|\log a|+C\,,\qquad\text{where $\kappa''\to0$ as $\kappa'\to0$}\,.
\end{equation}
Observe that, for all $k\ge0$, 
$$
\int\Big(\varphi-\int\varphi d\mu_a\Big)
\Big(\varphi-\int\varphi d\mu_a\Big)\circ T_a^kd\mu_a
=\int\varphi\varphi\circ T_a^kd\mu_a-\Big(\int\varphi d\mu_a\Big)^2\,.
$$
We get
\begin{multline*}
\sigma_a(\varphi)^2-\sigma_0(\varphi)^2
= \int\varphi^2(h_a-h_0)dm-\Big(\int\varphi d\mu_a\Big)^2\\
+2\sum_{k=1}^{k_0}
\Big(\int\varphi\varphi\circ T_a^kd\mu_a-\int\varphi\varphi\circ T_0^kd\mu_0
-\Big(\int\varphi d\mu_a\Big)^2\Big)+O(a^{\kappa'})\,.
\end{multline*}
By \eqref{eq.densitydiff}, we immediately get that the absolute value of the 
first two terms on the right hand side and of the last term in the sum
are bounded above by a constant (depending only on $\|\varphi\|_\infty$) times $a^\kappa$.
Regarding the remaining two integrals, again by \eqref{eq.densitydiff}, we have
$$
\int\varphi\varphi\circ T_a^kd\mu_a-\int\varphi\varphi\circ T_0^kd\mu_0
=\int\varphi\left(\varphi\circ T_a^k-\varphi\circ T_0^k\right)d\mu_0
+O\left(\|\varphi\|_\infty^2a^\kappa\right)\,.
$$
In order to bound the integral on the right hand side, we need the following sublemma.

\begin{sublemma}
\label{sl.J}
For all $a,a'\in[0,\epsilon']$ and $k\ge1$, there exists 
a set of intervals $\PP_k(a,a')$ such that for each $J\in\PP_k(a,a')$ there exist 
$\om\in\PP_k(a)$ and $\om'\in\PP_k(a')$ such that $J=\om\cap\om'$ and $\om$ and 
$\om'$ have the same combinatorics (up to iteration $k-1$). Furthermore, 
\begin{equation}
\label{eq.slJ}
\Big|\bigcup_{J\in\PP_k(a,a')}J\Big|
\ge|\supp\mu_a|- C\Big(\frac{p_1\Lambda}\lambda\Big)^k|a-a'|^\alpha\,.
\end{equation}
(Recall that $p_1$ is the number of elements in $\PP_1(a)$.)
\end{sublemma}

\begin{proof}
We show inductively in $k$ that 
\begin{equation}
\label{eq.slJi}
\Big|\bigcup_{J\in\PP_k(a,a')}J\Big|\ge|\supp\mu_a|
-C\sum_{j=1}^kp_1^j\lambda^{-(j-1)}\Lambda^{j-1}|a-a'|^\alpha\,.
\end{equation}
This immediately implies \eqref{eq.slJ}.

Recall that the properties~(i) and (iii) in the beginning of Section~\ref{s.uniform} asserts 
that the boundary points of the elements in $\PP_1(a)$ are
$\alpha$-H\"older continuous and the partition points $b_j(a)$, $0\le j\le p_0$, are 
Lipschitz continuous in $a$. This 
immediately shows \eqref{eq.slJi} for $k=1$. 
Let $k>1$ and assume the assertion holds for $k-1$. 
For $J_0\in\P_{k-1}(a,a')$, let $\om\in\PP_k(a)|J_0$ and $j=j(\om)$ such that 
$T_a^{k-1}(\om)\subset (b_{j-1}(a),b_j(a))$. By \eqref{eq.xaa} in
Lemma~\ref{l.distortion} and 
by the Lipschitz continuity of $b_{j-1}(a)$ and $b_j(a)$,
we derive
\begin{align*}
|T_{a'}^{k-1}(\om)\cap (b_{j-1}(a'),b_j(a'))|&\ge|T_{a'}^{k-1}(\om)|
-2L|a-a'|-2C\Lambda^{k-1}|a-a'|\\
&\ge |T_{a'}^{k-1}(\om)|-3C\Lambda^{k-1}|a-a'|\,.
\end{align*}
If the right hand side is positive, then 
we find $\om'\in\PP_k(a')|J_0$ with the same 
combinatorics as $\om$. Furthermore, 
by the distortion 
estimate \eqref{eq.distortion2} in Lemma~\ref{l.distortion} (where we set $a_1=a_2=a'$), 
we find a constant $C$ (independent on $\om$) such that 
$$
|\om\cap\om'|\ge|\om|-C\lambda^{-(k-1)}\Lambda^{k-1}|a-a'|\,.
$$
Since there are maximal $p_1^{k-1}$ elements in $\PP_{k-1}(a,a')$ and 
maximal $p_1$ elements in $\PP_k(a)|J_0$, we derive that 
$$
\Big|\bigcup_{J\in\PP_k(a,a')}J\Big|
\ge\Big|\bigcup_{J_0\in\PP_{k-1}(a,a')}J_0\Big|
-Cp_1^k\lambda^{-(k-1)}\Lambda^{k-1}|a-a'|\,.
$$
By the induction assumption, this concludes the proof of \eqref{eq.slJi}.
\end{proof}

Observe that \eqref{eq.k0} implies that $a$ is bounded above by a constant times 
$e^{-k_0/\kappa''}$. Combined with \eqref{eq.xaa} in 
Lemma~\ref{l.distortion}, if $\kappa''\le1/2\log\Lambda$, 
we derive that for all $J\in\PP_k(0,a)$, $k\le k_0$,
$$
|T_a^k(x)-T_0^k(x)|\le C\Lambda^ka\le C^2a^{1/2}\,,
\qquad\forall x\in J\,.
$$
If in addition $\kappa''\le\min(\alpha/4\log(p_1),\alpha/2\log(p_1\Lambda))$, it follows, for $k\le k_0$,
\begin{align*}
\int_{\supp(\mu_0)}\left|\varphi\circ T_a^k-\varphi\circ T_0^k\right|dy
&\le \sum_{J\in\PP_k(0,a)}\int_J\left|\varphi\circ T_a^k-\varphi\circ T_0^k\right|dy
+Cp_1^{k_0}\Lambda^{k_0}a^\alpha\\
&\le\#\{J\in\PP_k(0,a)\}\int\osc(\varphi,C^2a^{1/2},y)dy+C^2a^{\alpha/2}\\
&\le p_1^{k_0}C^{2\alpha}a^{\alpha/2}|\varphi|_\alpha+C^2a^{\alpha/2}
\le\tilde Ca^{\min(\alpha/4,\alpha/2)}\,.
\end{align*}
Altogether, for $\kappa'$ sufficiently small, there exists a constant $C$ (depending on 
$\|\varphi\|_\infty$ and $\|\varphi\|_\alpha$) such that
$$
\left|\sigma_a(\varphi)^2-\sigma_0(\varphi)^2\right|
\le Ca^\kappa+Ck_0\left(\|h_0\|_\infty a^{\min(\alpha/4,\alpha/2)}
+a^{2\kappa}\right)+Ca^{\kappa'}\le C^2a^\kappa\,,
$$
where in the last inequality we possibly have to decrease $\kappa>0$. 

If $a'\neq0$ observe that by \eqref{eq.uniformLY} it follows 
$\|h_a\|_{\alpha}=\lim_{n\to\infty}\|\LL_a^nh_a\|_{\alpha}\le C\|h_a\|_{L^1}$, 
and by \eqref{eq.densitydiff} we conclude that 
\begin{equation}
\label{eq.ha}
\sup_{a\in[0,\epsilon']}\|h_a\|_{\alpha}<\infty\,.
\end{equation} 
Combined with \eqref{eq.inftyalpha}, this
ensures that the constant $C$ is uniform in $a'$.
This concludes the proof of Lemma~\ref{l.varregularity}.
\end{proof}


\section{Switching locally from the parameter to the phase space}
\label{s.switch}
The aim of this section is to prove the following Proposition~\ref{p.mainestimate} which is the 
main estimate needed in verifying a law of large numbers for the squares of the blocks 
defined in the following Section~\ref{s.asip} (see Lemma~\ref{l.LLNy} therein). 
Its proof is given in the end of this section. Recall that in Theorem~\ref{t.main}, 
we assume $\sigma_0(\varphi)>0$. Hence, by Lemma~\ref{l.varregularity}, we 
find a constant $\epsilon>0$ so that $\sigma_a(\varphi)>0$, for all $a\in[0,\epsilon]$.  
(Note that by the assumption in Lemma~\ref{l.varregularity} the present constant $\epsilon$ 
is smaller than the constant $\epsilon'$ in Proposition~\ref{p.udc} which ensures that 
we can apply this proposition in the following.)
Let 
\begin{equation}
\label{eq.lambda0}
\lambda_0=\min(\lambda^{\min(\alpha/3,\kappa/2)},\rho^{-1/2})>1\,,
\end{equation}
where $\kappa>0$ is so small as in \eqref{eq.densitydiff} and Lemma~\ref{l.varregularity}.
Fix $\eta>0$ so small that
\begin{equation}
\label{eq.eta}
\Big(\frac{p_1\Lambda}{\lambda}\Big)^{\eta}\le\lambda_0\,.
\end{equation}
The expectation $E(\xi)$ of a function $\xi(a)$ is the integral 
$\epsilon^{-1}\int_0^\epsilon\xi(a) da$.

\begin{proposition}
\label{p.mainestimate}
There exists a constant $C$ (depending essentially only on $\varphi$ and the constants 
in the uniform exponential decay of correlation of the family $T_a$) such that 
\begin{equation}
\label{eq.mainestimate}
\Big|E\Big(\sum_{k=m}^{m+n-1}\xi_k(a)\Big)^2-n\Big|\le C\,,\quad\forall\ m
\ \text{and }1\le n\le\eta m/2\,.
\end{equation}

Furthermore, for $m$ and $n$ as in \eqref{eq.mainestimate} and $v=m-m^{1/4}$, if 
$\om\in\PP_v$ such that $\lambda_0^{-m^{1/4}}\le|x_v(\om)|\le n^{-3/\alpha}$ then 
\begin{equation}
\label{eq.mainestimatelocal}
\Big|E\Big[\Big(\sum_{k=m}^{m+n-1}\xi_k(a)\Big)^2\mid\{a\in\om\}\Big]-n\Big|\le C\,.
\end{equation}
\end{proposition}

\begin{remark}
With some more effort inequalities \eqref{eq.mainestimate} and \eqref{eq.mainestimatelocal} 
can be proven to hold for all $n\ge1$ (the argument uses a similar construction as in 
\eqref{eq.xxxx} below; the upper bound of $|x_v(\om)|$ in \eqref{eq.mainestimatelocal} 
can be replaced, e.g., by $2\lambda_0^{-m^{1/4}}$).
\end{remark}

The following lemma provides us with a tool to switch locally from the parameter space 
to the phase space. This can then be used in the proof of Proposition~\ref{p.mainestimate} 
to gain informations about the sequence $x_j$ on the parameter space 
by considering the iterations $T_{a_0}^j$ on the phase space for a fixed 
parameter value $a_0$. Recall the definition \eqref{eq.varphia} of $\varphi_a$.

\begin{lemma}[Switching locally from parameter to phase space]
\label{l.switch}
There exists a constant $C$ such that the following holds. Let $v,v_1,...,v_{\ell_0}$, 
$1\le\ell_0\le4$, be integers satisfying 
$$
1\le v\le v_1,...,v_{\ell_0}\le v+\eta v\,.
$$
Let $\om$ be an interval such that there exists 
$\ome\in\PP_v$ with $\om\subset\ome$ and $|x_v(\om)|\ge\lambda_0^{-v}$.  
For all $a_0\in\om$ we have 
\begin{equation}
\label{eq.switch}
\frac1{|x_v(\om)|}\int_{x_v(\om)}\bigg|\prod_{\ell=1}^{\ell_0}\xi_{v_\ell}(x_v|_\om^{-1}(y))
-\prod_{\ell=1}^{\ell_0}\varphi_{a_0}(T_{a_0}^{v_\ell-v}(y))\bigg|dy
\le C\lambda_0^{-v}\,.
\end{equation}
\end{lemma}

\begin{proof}
In order to prove Lemma~\ref{l.switch}, we need an ingredient similar to the one 
provided by Sublemma~\ref{sl.J}. The difference here is that we compare the 
partitions on the parameter space with the partitions on the phase space.

\begin{sublemma}
\label{sl.*J}
Let $v$, $\om$, and $a_0$ be as in the assertion of Lemma~\ref{l.switch}, and let 
$v\le\nu\le v+\eta v$. There exists 
a set of intervals $\PP_{v,\nu}(\om,a_0)$ such that for each $J\in\PP_{v,\nu}(\om,a_0)$ there exist 
$\om_1\in\PP_\nu|\om$ and $\om_2\in\PP_{\nu-v}(a_0)|x_v(\om)$ such that 
$J=x_v(\om_1)\cap\om_2$ and $x_v(\om_1)$ and 
$\om_2$ have the same combinatorics, i.e., for $a\in\om_1$ and $x\in\om_2$, 
$x_{v+i}(a)$ and $T_{a_0}^i(x)$ have the same combinatorics for $0\le i<\nu-v$. 
Furthermore, 
\begin{equation}
\label{eq.sl*J}
\Big|\bigcup_{J\in\PP_{v,\nu}(\om,a_0)}J\Big|\ge|x_v(\om)|(1-C\lambda_0^{-v})\,.
\end{equation}
\end{sublemma}

\begin{proof}
The proof is similar to the proof of Sublemma~\ref{sl.J}. 
For $\nu=v$ there is nothing to show (by definition $x_v(\om)=\PP_0|x_v(\om)$). Henceforth, we 
assume $\nu>v$.
Given $v$, we show inductively in $\nu>v$ that 
\begin{equation}
\label{eq.sl*Ji}
\Big|\bigcup_{J\in\PP_{v,\nu}(\om,a_0)}J\Big|\ge|x_v(\om)|
-C\sum_{j=v+1}^\nu p_1^{j-v}\lambda^{-(j-v-1)}\Lambda^{j-v-1}|\om|^\alpha\,.
\end{equation}
If $\nu\le v+\eta v$, 
inequality \eqref{eq.sl*Ji} implies then 
\begin{align*}
\Big|\bigcup_{J\in\PP_{v,\nu}(\om,a_0)}J\Big|&\ge|x_v(\om)|
-C^2p_1^{\nu-v}\lambda^{-(\nu-v)}\Lambda^{\nu-v}|\om|^\alpha\\
&\ge|x_v(\om)|
-C^2\Big(\frac{p_1\Lambda}{\lambda}\Big)^{\eta v}\lambda^{-\alpha v}
\ge|x_v(\om)|(1-C^3\lambda_0^{-v})\,,
\end{align*}
where in the last inequality we used the condition \eqref{eq.eta} on $\eta$ and the fact that 
$|x_v(\om)|\ge\lambda_0^{-v}$ (we used also that $|\om|\le C\lambda^{-v}$ which 
follows from Lemma~\ref{l.transversality}). 
This concludes the proof of Sublemma~\ref{sl.*J}.

Since, by properties~(i) and (iii) in Section~\ref{s.uniform}, the boundary points of 
$\om\in\PP_1(a_0)|x_v(\om)$ are $\alpha$-H\"older continuous
and the partition points $b_0(a),b_1(a),...,b_{p_0}(a)$ are Lipschitz continuous in $a$, this 
immediately shows \eqref{eq.sl*Ji} for $\nu=v+1$. 
Let $\nu>v+1$ and assume the assertion holds for $\nu-1$. 
For $J_0\in\P_{v,\nu-1}(\om,a_0)$, let $\om_2\in\PP_{\nu-v}(a_0)|J_0$ and $j=j(\om_2)$ such that 
$T_{a_0}^{\nu-v-1}(\om_2)\subset (b_{j-1}(a_0),b_j(a_0))$. 
By \eqref{eq.xaa} in Lemma~\ref{l.distortion} and 
by the Lipschitz continuity of $b_{j-1}(a)$ and $b_j(a)$, 
we derive
\begin{multline*}
|x_{\nu-1}(x_v|_\om^{-1}(\om_2))\cap (b_{j-1}(a),b_j(a'))|\\
\ge|x_{\nu-1}(x_v|_\om^{-1}(\om_2))|
-2L|\om|-2C\Lambda^{\nu-v-1}|\om|\,,\quad\forall a,a'\in\om\,.
\end{multline*}
If the right hand side is positive for an appropriate choice of $a,a'\in\om$, 
then we find $\om_1\in\PP_\nu|\om$, where $x_v(\om_1)$ 
and $\om_2$ have the same combinatorics. Furthermore, by the distortion 
estimate \eqref{eq.distortion2} in Lemma~\ref{l.distortion}, 
we find a constant $C$ such that 
\begin{align*}
|x_v(\om_1)\cap\om_2|&\ge|\om_2|
-C\frac{|\om_2|}{|x_{\nu-1}(x_v|_\om^{-1}(\om_2))|}\Lambda^{\nu-v-1}|\om|\\
&\ge|\om_2|
-C^2\lambda^{-(\nu-v-1)}\Lambda^{\nu-v-1}|\om|\,,
\end{align*}
where in the last inequality we used \eqref{eq.distortion3}.
Since there are maximal $p_1^{\nu-v-1}$ elements in $\PP_{v,\nu-1}(\om,a_0)$ and 
maximal $p_1$ elements in $\PP_{\nu-v}(a_0)|J_0$, we derive that 
$$
\Big|\bigcup_{J\in\PP_{v,\nu}(\om,a_0)}J\Big|
\ge\Big|\bigcup_{J_0\in\PP_{v,\nu-1}(\om,a_0)}J_0\Big|
-C^2p_1^{\nu-v}\lambda^{-(\nu-v-1)}\Lambda^{\nu-v-1}|\om|\,.
$$
By the induction assumption, this concludes the proof of \eqref{eq.slJi}.
\end{proof}

Recall that $\xi_{v_1}(a)=\varphi_a(x_{v_1}(a))$ (see \eqref{eq.xigen}). 
For $a=x_v|_\om^{-1}(y)$ and $a_0\in\om$, 
we write
\begin{multline*}
\xi_{v_1}(a)-\varphi_{a_0}(T_{a_0}^{v_1-v}(x_v(a)))\\
=\varphi_a(x_{v_1}(a))
-\varphi_{a_0}(x_{v_1}(a))
+\varphi_{a_0}(x_{v_1}(a))
-\varphi_{a_0}(T_{a_0}^{{v_1}-v}(x_v(a)))\,.
\end{multline*}
By Lemma~\ref{l.varregularity} and \eqref{eq.densitydiff}, we easily see that 
the difference of the first two terms on the right hand side is bounded from above 
by a constant times $|\om|^\kappa$. To estimate the integral over 
the difference of the last two terms we use the partition given by 
Sublemma~\ref{sl.*J}. First, observe that, by Lemma~\ref{l.varregularity} 
and \eqref{eq.inftyalpha}, we find a constant $C$ only 
dependent on $\varphi$ (and, in particular, not on $a$) so that 
\begin{equation}
\label{eq.alphal1infty}
\max(\|\varphi_a\|_\alpha,\|\varphi_a\|_{L^1},\|\varphi_a\|_\infty)
\le C\|\varphi\|_\alpha\le C^2\,,\qquad\forall a\in[0,\epsilon]\,.
\end{equation}
If $J\in\PP_{v,v_1}(\om,a_0)$ and $y\in J$, then by 
\eqref{eq.xaa} in Lemma~\ref{l.distortion} we have 
$$
\big|x_{v_1}(x_v|_\om^{-1}(y))-T_{a_0}^{{v_1}-v}(y)\big|\le C\Lambda^{{v_1}-v}|\om|\,,
$$
which implies that 
\begin{align*}
&\int_{x_v(\om)}\big|\varphi_{a_0}(x_{v_1}(x_v|_\om^{-1}(y)))
-\varphi_{a_0}(T_{a_0}^{{v_1}-v}(y))\big|dy\\
&\le\sum_{J\in\PP_{v,{v_1}}(\om,a_0)}
\int_{J}\osc(\varphi_{a_0},C\Lambda^{{v_1}-v}|\om|,T_{a_0}^{{v_1}-v}(y))dy
+C|x_v(\om)|\lambda_0^{-v}\\
&\le\sum_{J\in\PP_{v,{v_1}}(\om,a_0)}
C\lambda^{-({v_1}-v)}
\int_{T_{a_0}^{{v_1}-v}(J)}\osc(\varphi_{a_0},C\Lambda^{{v_1}-v}|\om|,z)dz
+C|x_v(\om)|\lambda_0^{-v}\\
&\le C^3\Big(\frac{p_1\Lambda^\alpha}{\lambda}\Big)^{{v_1}-v}|\om|^\alpha
+C|x_v(\om)|\lambda_0^{-v}\,.
\end{align*}
Altogether, we obtain (recall \eqref{eq.alphal1infty})
\begin{multline*}
\frac1{|x_v(\om)|}\int_{x_v(\om)}
\Big|\xi_{v_1}(x_v|_\om^{-1}(y))-\varphi_{a_0}(T_{a_0}^{{v_1}-v}(y))\Big|
\prod_{\ell=2}^{\ell_0}|\xi_{v_\ell}(x_v|_\om^{-1}(y))|dy\\
\le C\Big(
\Big(\frac{p_1\Lambda^\alpha}{\lambda}\Big)^{\eta v}\frac{|\om|^\alpha}{|x_v(\om)|}
+\lambda_0^{-v}+\frac{|\om|^\kappa}{|x_v(\om)|}\Big)
\le C^2\lambda_0^{-v}\,,
\end{multline*}
where
in the last inequality we used the assumption that $|x_v(\om)|\ge\lambda_0^{-v}$, the 
definition \eqref{eq.lambda0} of $\lambda_0$, and 
the condition \eqref{eq.eta} on $\eta$. 
Then, similarly we derive
\begin{multline*}
\frac1{|x_v(\om)|}\int_{x_v(\om)}|\varphi_{a_0}(T_{a_0}^{{v_1}-v}(y))|
\Big|\xi_{v_2}(x_v|_\om^{-1}(y))-\varphi_{a_0}(T_{a_0}^{{v_2}-v}(y))\Big|\\
\prod_{\ell=3}^{\ell_0}|\xi_{v_\ell}(x_v|_\om^{-1}(y))|dy
\le C\lambda_0^{-v}\,,
\end{multline*}
and so on. This concludes the proof of Lemma~\ref{l.switch}.
\end{proof}

\begin{corollary}
\label{c.switch}
There exists a constant $C$ such that the following holds. Let $v,v_1,...,v_{\ell_0}$, 
$1\le\ell_0\le4$, be positive integers as in the assertion of Lemma~\ref{l.switch} and 
let $\om$ be an interval such that there exists 
$\ome\in\PP_v$ with $\om\subset\ome$ and $|x_v(\om)|\ge\lambda_0^{-v}$.  
For all $a_0\in\om$ we have 
$$
\frac1{|\om|}\int_\om\prod_{\ell=1}^{\ell_0}\xi_{v_\ell}(a)da
=\frac{1+O(|x_v(\om)|^\alpha)}{|x_v(\om)|}\int_{x_v(\om)}
\prod_{\ell=1}^{\ell_0}\varphi_{a_0}(T_{a_0}^{v_\ell-v}(y))dy+O(\lambda_0^{-v})\,.
$$
\end{corollary}

\begin{proof}
Doing the change of variables $y=x_v(a)$, $a\in\om$,
by the distortion estimate \eqref{eq.distortion1} in Lemma~\ref{l.distortion}, we derive 
\begin{equation}
\label{eq.switchdist}
\frac1{|\om|}\int_\om\prod_{\ell=1}^{\ell_0}\xi_{v_\ell}(a)da
=\frac{1+O(|x_v(\om)|^\alpha)}{|x_v(\om)|}
\int_{x_v(\om)}\prod_{\ell=1}^{\ell_0}\xi_{v_\ell}(x_v|_\om^{-1}(y))dy\,.
\end{equation}
Applying Lemma~\ref{l.switch} concludes the proof.
\end{proof}

\subsection{Proof of Proposition \ref{p.mainestimate}}
We are going to show \eqref{eq.mainestimatelocal}.
Let $\om$ be as in the assertion.
We write
$$
\frac1{|\om|}\int_\om\Big(\sum_{k=m}^{m+n-1}\xi_k\Big)^2
=\frac1{|\om|}\sum_{k=m}^{m+n-1}
\Big(\int_\om\xi_k^2+2\sum_{\ell=k+1}^{m+n-1}\int_\om\xi_k\xi_\ell\Big)\,.
$$
Hence, in order to prove \eqref{eq.mainestimatelocal}, 
it is sufficient to show that there is a constant $C$ such that
\begin{equation}
\label{eq.firstterm}
\Big|1-\frac1{|\om|}\int_\om\Big(\xi_k^2+2\sum_{\ell=k+1}^{m+n-1}\xi_k\xi_l\Big)\Big|
\le C(m+n-k)^{-2}\,.
\end{equation}
Let $a_0\in\om$.
By Corollary~\ref{c.switch}, we have 
\begin{align*}
&\frac1{|\om|}\int_\om\Big(\xi_k^2+2\sum_{\ell=k+1}^{m+n-1}\xi_k\xi_l\Big)\\
&=\frac1{|x_v(\om)|}
\int_{x_v(\om)}\Big(\varphi_{a_0}^2\circ T_{a_0}^{k-v}
+2\sum_{\ell=k+1}^{m+n-1}\varphi_{a_0}\circ T_{a_0}^{k-v}\varphi_{a_0}\circ T_{a_0}^{\ell-v}\Big)\\
&\quad+O\big((m+n-k)(\lambda_0^{-(m-m^{1/4})}+|x_v(\om)|^\alpha)\big)\,.
\end{align*}
Proposition~\ref{p.udc} gives (recall also \eqref{eq.alphal1infty}), for all $\ell\ge k$,
\begin{multline*}
\int_{x_v(\om)}\varphi_{a_0}\circ T_{a_0}^{k-v}\varphi_{a_0}\circ T_{a_0}^{\ell-v}dm
=|x_v(\om)|\int\varphi_{a_0}\varphi_{a_0}\circ T_{a_0}^{\ell-k}d\mu_{a_0}
+O(\rho^{k-v})\,.
\end{multline*}
By the normalisation \eqref{eq.varphinormalised} and applying 
once more Proposition~\ref{p.udc}, we have
$$
\int\Big(\varphi_{a_0}^2
+2\sum_{\ell=k+1}^{m+n-1}\varphi_{a_0}\varphi_{a_0}\circ T_{a_0}^{\ell-k}\Big)d\mu_{a_0}
=1+O(\rho^{m+n-k})\,.
$$
Hence, we conclude 
\begin{multline*}
\Big|1-\frac1{|\om|}\int_\om\Big(\xi_k^2+2\sum_{\ell=k+1}^{m+n-1}\xi_k\xi_l\Big)\Big|\\
\le C(m+n-k)\big(\lambda_0^{-(m-m^{1/4})}+|x_v(\om)|^\alpha+\rho^{k-v}/|x_v(\om)|\big)+C\rho^{m+n-k}\,.
\end{multline*}
Regarding \eqref{eq.firstterm}, the first and last term on the right hand side are fine, and also 
the second term since by assumption $|x_v(\om)|\le n^{-3/\alpha}$. For the 
remaining term we use the lower bound $|x_v(\om)|\ge\lambda_0^{-m^{1/4}}$ which gives 
(recall the definition of $\lambda_0$ in \eqref{eq.lambda0})
$$
\rho^{k-v}/|x_v(\om)|\le\rho^{m^{1/4}}\lambda_0^{m^{1/4}}\le\lambda_0^{-m^{1/4}}\,.
$$

In order to prove Proposition~\ref{p.mainestimate}, it is only left to prove \eqref{eq.mainestimate} 
which follows now easily from \eqref{eq.mainestimatelocal} combined with 
Lemma~\ref{l.goodpartition} in which we take $d_v=Ce^{-v^{1/5}}$ (where $C$ is taken 
so that $d_v\ge\lambda_0^{-m^{1/4}}$).
Recall that by Lemma~\ref{l.goodpartition}, for each $v\ge1$, 
there is an exceptional set $E_v\subset\PP_v$ so that $|E_v|\le Cd_v^{1/2}$ and 
$|x_v(\om)|\ge d_v$ for all $\om\in\PP_v\setminus E_v$.
Let $\PP_v^*$ be a refinement of the partition $\PP_v\setminus E_v$ so that 
for $\om\in\PP_v^*$ we have $d_v\le|x_v(\om)|\le2d_v$. 
Since $d_v\ge\lambda_0^{-m^{1/4}}$, by \eqref{eq.mainestimatelocal}, we obtain
\begin{align*}
E\Big(\sum_{k=m}^{m+n-1}\xi_k(a)\Big)^2
&=O(|E_v|)+\sum_{\om\in\PP_v^*}|\om|(n+O(1))/\epsilon\\
&=\frac{\epsilon-|E_v|}\epsilon n+O(1)=n+O(d_v^{1/2}n)+O(1)\,.
\end{align*}
Since $d_v^{1/2}n\le Ce^{-m^{1/5}/C}m=o(1)$, this concludes the proof of 
\eqref{eq.mainestimate} and, thus, the proof of Proposition~\ref{p.mainestimate}.


\section{Proof of Theorem~\ref{t.main} via Skorokhod's representation theorem}
\label{s.asip}
As mentioned in the introduction, in order to prove Theorem~\ref{t.main}, 
we go along the classical, probabilistic approach in \cite{ps}. 
It consists in rearranging the Birkhoff sum as a sum of 
blocks of polynomial size where we then approximate the blocks by a martingale 
and apply Skorokhod's representation 
theorem to it. The optimal power of the 
polynomial size of the blocks in our setting is $2/3$ 
which gives then an error exponent $\gamma>2/5$ in the almost sure invariance principle
in Theorem~\ref{t.main}. 
Being familiar with the technique in \cite{ps}, it is natural to ask if the error exponent could 
be decreased to $\gamma>1/3$: If one considers a fixed dynamical system as, e.g., in 
\cite{hk}, then one could take $1/2$ as the power of the polynomial size of the block and 
when separating these blocks by small blocks of logarithmic (or very small polynomial) size 
then this would lead to an error exponent $\gamma>1/3$. However, in our setting  
the estimate~\eqref{eq.43} below is not good enough to be able to establish an error 
exponent $\gamma>1/3$, and we don't know how to improve this estimate. 
In the recent work \cite{gouezel}, Gou\"ezel uses spectral methods to show 
an almost sure invariance principle and he obtains remarkable error estimates which are 
independent on the dimension of the process. 
For example for the maps studied in \cite{hk} he gets the error exponent $\gamma>1/4$. 
However, we didn't 
find an easy way to apply these spectral methods to our setting. 
The strategy via Skorokhod's representation theorem is also convenient here 
because of its simplicity. Nevertheless, since our setting is rather special, 
we have to go step by step through the method of 
building blocks and approximating by martingales. 
In particular, we cannot apply directly the main statement in \cite[Chapter 7]{ps} 
since the functions $\xi_i$ are maps on the parameter space where the concept of 
invariant measures does not make any sense and we are not able to verify 
nor to formulate an analog of a strong mixing condition (cf. \cite[7.1.2]{ps}) in our setting. 
However, they are statements in \cite{ps} which we can take over more or less one to one. 
This will keep this section of a reasonable length.

\subsection{Building the blocks}
\label{ss.blocks}
Fix a constant $\epsilon>0$ as in the beginning of Section~\ref{s.switch}. This ensures 
that we can apply all the results in Sections~\ref{s.prel} and \ref{s.switch}.
Take $\delta>0$ sufficiently small (to be determined later on; see, e.g., 
the proof of Lemma~\ref{l.LLNY} below).
We approximate the functions $\xi_i:[0,\epsilon]\to[0,1]$, $i\ge1$, 
by stepfunctions $\chi_i$. 
In order to do that, 
we introduce the $\sigma$-fields $\FF_i$ which are generated by the 
intervals in $\PP_{r_i}(=\PP_{r_i}|[0,\epsilon])$ where $r_i=i+[i^{\delta}]$. 
Observe that, by \eqref{eq.distortion3} and \eqref{eq.la}, 
\begin{equation}
\label{eq.ri}
|x_i(\om)|\le C\lambda^{-i^{\delta}}\,,\qquad\forall \om\in\PP_{r_i}\,.
\end{equation}
The stepfunctions $\chi_i$ are defined as
$\chi_i=E(\xi_i\mid\FF_i)$.
Recall the constants $\rho$ in Proposition~\ref{p.udc}, $\lambda_0$ in \eqref{eq.lambda0}, 
and $\eta$ in 
\eqref{eq.eta}.
We introduce a constant $\rho_0$ defined as
\begin{equation}
\label{eq.rho0}
\rho_0=\max(\rho^{\eta/(1+\eta)},\lambda_0^{-1/(1+\eta)})<1\,.
\end{equation}
We have the following basic properties.

\begin{lemma}
\label{l.chiapprox}
For almost every $a\in[0,\epsilon]$, we have
\begin{equation}
\label{eq.chia1}
|\xi_i(a)-\chi_i(a)|\le \lambda^{-\alpha i^{\delta}/8}\,,\qquad\text{for all but finitely many } i\ge1\,.
\end{equation}

Furthermore, there exists a constant $C$ such that for all $i\ge1$ and $j\ge0$ there exists an 
exceptional set of intervals $E_{i,j}$ so that for a.e. $a\in[0,\epsilon]\setminus E_{i,j}$
\begin{equation}
\label{eq.chia2}
|E(\xi_{i+j}\mid\FF_i)(a)|=|E(\chi_{i+j}\mid\FF_i)(a)|\le C\min(1,\rho_0^{j-2i^{\delta}})\,,
\end{equation}
where $E_{i,j}=\emptyset$, for $j\le2i^{\delta}$, and 
$|E_{i,j}|\le C\rho_0^{(j-i^{\delta})/2}$, otherwise.
\end{lemma}

\begin{proof}
We show first \eqref{eq.chia1}.
Let $\om\in\PP_i$ and fix an arbitrary parameter $a_\om$ in $\om$.
For $\ome\in\PP_{r_i}|\om$ and $a_0\in\ome$, by the definition of $\varphi_a$, \eqref{eq.densitydiff}, 
and Lemma~\ref{l.varregularity}, we have
$$
\chi_i(a_0)=\frac1{|\ome|}\int_{\ome}\xi_i(a)da=\frac1{|\ome|}
\int_{\ome}\varphi_{a_\om}(x_i(a))da+O(|\om|^\kappa)\,,
$$
which implies that, for a.e. $a_0\in\ome$,
$$
|\varphi_{a_\om}(x_i(a_0))-\chi_i(a_0)|
\le\esssup_{a\in\ome}\varphi_{a_\om}(x_i(a))-\essinf_{a\in\ome}\varphi_{a_\om}(x_i(a))
+C|\om|^\kappa
$$
Recall the estimate \eqref{eq.ri}. Let 
$E_i=\{\om\in\PP_i\mid|x_i(\om)|\le\lambda^{-\alpha i^\delta/2}\}$.
We get
\begin{align*}
&|\{|\xi_i-\chi_i|\ge \lambda^{-\alpha i^{\delta}/8}\}|\\
&\le \lambda^{\alpha i^{\delta}/8}\Big[C|E_i|
+\sum_{\om\in\PP_i\setminus E_i}\int_\om|\xi_i(a)-\chi_i(a)|da\Big]\\
&\le \lambda^{\alpha i^{\delta}/8}\Big[C|E_i|
+\sum_{\om\in\PP_i\setminus E_i}\frac{C|\om|}{|x_i(\om)|}
\Big(\int_{x_i(\om)}|\varphi_{a_\om}(y)-\chi_i(x_i|_\om^{-1}(y))|dy+C|\om|^\kappa\Big)\Big]\\
&\le \lambda^{\alpha i^{\delta}/8}\Big[C|E_i|
+\sum_{\om\in\PP_i\setminus E_i}\frac{C|\om|}{|x_i(\om)|}
\Big(\int_0^1\osc(\varphi_{a_\om},C\lambda^{-i^{\delta}},y)dy+2C|\om|^\kappa\Big)\Big]\,.
\end{align*}
The integral is bounded by a constant times 
$\lambda^{-\alpha i^{\delta}}$ (recall \eqref{eq.alphal1infty}). By Lemma~\ref{l.goodpartition}, 
we have $|E_i|\le C\lambda^{-\alpha i^\delta/4}$. It follows 
\begin{equation}
\label{eq.chiahelp}
|\{|\xi_i-\chi_i|\ge \lambda^{-\alpha i^{\delta}/8}\}|
\le C\lambda^{-\alpha^\delta/8}\,.
\end{equation}
By Borel-Cantelli this concludes the proof of \eqref{eq.chia1}.

We turn to the proof of \eqref{eq.chia2}. If $j\le2i^{\delta}$, there is nothing to prove. 
If $j\ge2i^{\delta}$, let $k=\max(r_i,(i+j)/(1+\eta))$. 
Denoting by $\widetilde \PP_k$ the $\sigma$-field generated by the intervals in $\PP_k$, 
observe that we have
$$
|E(\xi_{i+j}\mid\FF_i)(a)|=|E(E(\xi_{i+j}\mid\widetilde\PP_k)\mid\FF_i)(a)|\,.
$$
Hence, in order to prove \eqref{eq.chia2}, it is sufficient to consider the terms
$$
\frac1{|\om|}\Big|\int_\om\xi_{i+j}(a)da\Big|\,,\qquad\om\in\PP_k\,.
$$
For $\om\in\PP_k$, we have, by \eqref{eq.distortion1}, 
$$
\frac1{|\om|}\Big|\int_\om\xi_{i+j}(a)da\Big|
\le \frac{C}{|x_k(\om)|}\Big|\int_{x_k(\om)}\xi_{i+j}(x_k|_\om^{-1}(y))dy\Big|\,.
$$
Regarding Lemma~\ref{l.switch}, 
we can only give a good estimate of the right hand side, if the image of $\om$ under $x_k$ 
is sufficiently large. Hence, we define the exceptional set
$$
E_{i,j}=\{\om\in\PP_k\mid|x_k(\om)|\le\rho_0^{j-i^{\delta}}\}\,.
$$
By the definition of $k$, we derive 
that $\rho_0^{j-i^\delta}$ is smaller than $\rho_0^{k^\delta/2}$ for $j\ge2 i^\delta$ small 
and smaller than $\rho_0^{\eta k}$ for $j$ large. In particular, this implies that 
$\rho_0^{j-i^\delta}$ is decaying stretched exponentially fast in $k$. 
Applying Lemma~\ref{l.goodpartition}, we derive that 
$|E_{i,j}|\le C\rho_0^{(j-i^{\delta})/2}$ for some constant $C$ (since the constants 
in the above two upper bounds for $\rho_0^{j-i^\delta}$ are uniform, 
the proof of Lemma~\ref{l.goodpartition} easily shows
that this constant $C$ can be chosen uniformly in $i$ and $j$). 
On the other hand, by the definition of $k$, $\rho_0^{j-i^\delta}$ is greater than 
$\rho_0^{(1+\eta)k}$ which in turn is, by the definition \eqref{eq.rho0} of $\rho_0$, 
greater than $\lambda_0^{-k}$.
In other words $|x_k(\om)|\ge\lambda_0^{-k}$, for $\om\in\PP_k\setminus E_{i,j}$, 
and we can apply \eqref{eq.switch} in Lemma~\ref{l.switch} which gives 
(observe that by the definition of $k$ we have $k\le i+j\le k+\eta k$)
$$
\frac1{|x_k(\om)|}\Big|\int_{x_k(\om)}
\xi_{i+j}(x_k|_\om^{-1}(y))-\varphi_{a_0}(T_{a_0}^{i+j-k}(y))dy\Big|
\le C\lambda_0^{-k}\le C\rho_0^{j-i^\delta}\,.
$$
By Proposition~\ref{p.udc} and \eqref{eq.alphal1infty}, we get
$$
\frac1{|x_k(\om)|}\Big|\int_{x_k(\om)}\varphi_{a_0}(T_{a_0}^{i+j-k}(y))dy\Big|\le C\rho^{i+j-k}\,.
$$
Since $i+j-k\ge\eta(j-i^\delta)/(1+\eta)$ for all $j\ge2i^\delta$ and for $k$ as defined above,  
by the definition \eqref{eq.rho0} of $\rho_0$, we conclude 
$$
\frac1{|\om|}\Big|\int_\om\xi_{i+j}(a)da\Big|\le C(\rho^{\eta(j-i^\delta)/(1+\eta)}+\rho_0^{j-i^\delta})
\le2C\rho_0^{j-i^\delta}\,,
$$
which concludes the proof of \eqref{eq.chia2}.
\end{proof}

We define blocks of integers $I_j$, $j\ge1$, inductively where $I_1=\{1\}$ and $I_j$ contains 
$[j^{2/3}]$ consecutive integers and there are no gaps between the blocks. 
For $j\ge1$, we set 
$$
y_j:=\sum_{i\in I_j}\chi_i\,.
$$
Let $M=M(N)$ denote the index of $y_j$ containing $\chi_N$. Observe that there 
exists a constant $C$ so that 
\begin{equation}
\label{eq.MN}
C^{-1}N^{3/5}\le M\le CN^{3/5}\,,\qquad\forall N\ge1\,.
\end{equation}
By \eqref{eq.chia1}, for a.e. $a\in[0,\epsilon]$, we find a constant $C(a)$ so that
\begin{equation}
\label{eq.lastterm}
\Big|\sum_{i=1}^N\xi_i(a)-\sum_{j=1}^My_j(a)\Big|\\
\le\sum_{i=1}^N|\xi_i(a)-\chi_i(a)|+C|I_M|\le C(a)+CN^{2/5}\,,
\end{equation}
for all $N\ge1$.
Hence, in order to prove Theorem~\ref{t.main} it is 
sufficient to consider the sum $\sum_{j=1}^My_j$. 

\subsection{Law of large numbers for $y_j^2$}
In this section we will prove the following key lemma. It is the main technical ingredient 
in the proof of Theorem~\ref{t.main}.

\begin{lemma}
\label{l.LLNy}
For a.e. $a\in[0,\epsilon]$, there exists a constant $C$ such that 
\begin{equation}
\label{eq.yj2}
\Big|N-\sum_{j=1}^My_j^2(a)\Big|\le CN^{2\gamma}\,,\qquad\forall N\ge1,
\end{equation}
(where $\gamma>2/5$ is the error exponent in Theorem~\ref{t.main}).
\end{lemma}

Before we start with the proof of Lemma~\ref{l.LLNy}, we recall a version of the strong law 
of large numbers by Gal and Koksma. Its proof is, e.g., given in \cite[Theorem A.1]{ps}.

\begin{theorem}[Gal-Koksma's strong law of large numbers]
\label{t.galkoksma}
Let $z_j$, $j\ge1$, be zero-mean random variables and assume that there exist a real 
number $q\ge1$ and a constant $C$ 
such that 
$$
E\Big(\sum_{j=m+1}^{m+n}z_j\Big)^2\le C((m+n)^q-m^q)\,,\quad\forall\ m\ge0\ \text{and}\ n\ge1\,.
$$
Then for all $\iota>0$, we have $\frac1{n^{q/2+\iota}}\sum_{j=1}^nz_j\to0$ almost surely.
\end{theorem}

\begin{proof}[Proof of Lemma~\ref{l.LLNy}]
In this proof we will mainly work with the original $\xi_i$ instead of their approximations $\chi_i$.
Let 
$$
w_j=\sum_{i\in I_j}\xi_i\,.
$$
Writing $y_j^2-w_j^2=(y_j+w_j)(y_j-w_j)$, by \eqref{eq.chia1}, we derive that 
$\sum_{j\ge1}|y_j^2-w_j^2|$ is almost surely finite.
Hence, it is sufficient to prove \eqref{eq.yj2} where $y_j$ is replaced by $w_j$.
Regarding the $w_j$'s we claim that it is sufficient to show that 
for all $\iota>0$ there is a constant $C$ such that 
\begin{equation}
\label{eq.galkoksma}
E\Big(\sum_{j=m+1}^{m+n}w_j^2-E w_j^2\Big)^2\le C((m+n)^{8/3+\iota}-m^{8/3+\iota})\,, 
\quad\forall m\ge0,\ n\ge1\,.
\end{equation}
Indeed, by the estimate \eqref{eq.mainestimate} in Proposition~\ref{p.mainestimate}, we 
have
$$
\Big|N-\sum_{j=1}^Mw_j^2\Big|\le CM+\Big|\sum_{j=1}^Mw_j^2-Ew_j^2\Big|\,.
$$
Hence, applying Theorem~\ref{t.galkoksma} to \eqref{eq.galkoksma}
and recalling \eqref{eq.MN}, 
concludes the proof of \eqref{eq.yj2} (where $y_j$ is replaced by $w_j$).

In the following we will prove \eqref{eq.galkoksma}. Observe that $(Ew_j^2)^2\le Ew_j^4$. 
We have
\begin{equation}
\label{eq.galkoksmapre}
E\Big(\sum_{j=m+1}^{m+n}w_j^2-E w_j^2\Big)^2
\le2\sum_{j=m+1}^{m+n}\Big(Ew_j^4+\sum_{k=j+1}^{m+n}|Ew_j^2w_k^2-Ew_j^2Ew_k^2|\Big)\,.
\end{equation}
We consider first $Ew_j^4$. For $\iota>0$ small, let 
$S=\{(v_1,v_2,v_3,v_4)\in I_j^4\mid 
v_1\le v_2\le v_3\le v_4\text{ and either $v_2-v_1\ge j^\iota$ or $v_4-v_3\ge j^\iota$}\}$.
We have
\begin{equation}
\label{eq.SSS}
\int w_j(a)^4da\le C\sum_{(v_1,v_2,v_3,v_4)\in S}\Big|
\int\prod_{\ell=1}^4\xi_{v_\ell}(a)da\Big|
+Cj^{4/3+2\iota}\,.
\end{equation}

Let $(v_1,...,v_4)\in S$. We consider first the case when $v_4-v_3\ge j^\iota$. 
In order to apply Lemma~\ref{l.switch} we have to get rid of partition elements with a 
too small image.
Let $E_{v_3}=\{\om\in\PP_{v_3}\mid|x_{v_3}(\om)|\ge\rho^{j^\iota/2}\}$. 
By \eqref{eq.MN} and Lemma~\ref{l.goodpartition}, the measure of $E_{v_3}$ is 
decaying stretched exponentially fast in $j$. 
For $\om\in\PP_{v_3}\setminus E_3$ and $a_0\in\om$, 
by equality~\eqref{eq.switchdist} (for $\ell_0=4$) combined with
 Lemma~\ref{l.switch} (for $\ell_0=1$), we derive
\begin{multline*}
\frac1{|\om|}\Big|\int_\om\prod_{\ell=1}^4\xi_\ell(a)da\Big|
\le\frac C{|x_{v_3}(\om)|}\Big| \int_{x_{v_3}(\om)}
\Big(\prod_{\ell=1}^3\xi_\ell(x_{v_3}|_\om^{-1}(y))\Big)
\varphi_{a_0}(T_{a_0}^{v_4-v_3}(y))dy\Big|\\
+C\lambda_0^{-v_3}\,.
\end{multline*}
Let $L_\ell=x_{v_\ell}\circ x_{v_3}|_\om^{-1}$, $1\le \ell\le3$. 
For $y\in x_{v_3}(\om)$, we have $\xi_{v_\ell}(x_{v_3}|_\om^{-1}(y))=\varphi_a(L_\ell(y))$ 
where $a=x_{v_3}|_\om^{-1}(y)$.
Hence, by Lemma~\ref{l.varregularity} and \eqref{eq.densitydiff}, we derive that 
$|\xi_{v_\ell}(x_{v_3}|_\om^{-1}(y))-\varphi_{a_0}(L_\ell(y))|
\le C|\om|^\kappa$.
It follows
\begin{multline}
\label{eq.xxxx}
\Big|\int\prod_{\ell=1}^4\xi_{v_\ell}(a)da\Big|
\le C|E_{v_3}|+\!\!\!\!\! \sum_{\om\in\PP_{v_3}\setminus E_{v_3}}\!\!\!|\om|\bigg[
\frac C{|x_{v_3}(\om)|}
\Big|\int_{x_{v_3}(\om)}\Big(\prod_{\ell=1}^3\varphi_{a_0}(L_\ell(y))\Big)\\
\varphi_{a_0}(T_{a_0}^{v_4-v_3}(y))dy\Big|+\frac {C|\om|^\kappa}{|x_{v_3}(\om)|}
+C\lambda_0^{-v_3}\bigg]\,.
\end{multline}
Observe that, by \eqref{eq.distortion3}, we have 
$|L_\ell'(y)|\le C\lambda^{-(v_3-v_1)}$, which implies
$|\chi_{x_{v_3}(\om)}\varphi_{a_0}\circ L_\ell|_\alpha\le C|\chi_{x_{v_\ell}(\om)}
\varphi_{a_0}|_\alpha$. 
Since $\|\chi_{x_{v_3}(\om)}\varphi_{a_0}\circ L_\ell\|_\infty
=\|\chi_{x_{v_\ell}(\om)}\varphi_{a_0}\|_\infty$, 
by \eqref{eq.inftyalpha}, 
$$
\|\chi_{x_{v_3}(\om)}\varphi_{a_0}\circ L_\ell\|_\alpha
\le C\|\chi_{x_{v_\ell}(\om)}\varphi_{a_0}\|_\alpha
\le C^2\|\varphi_{a_0}\|_\alpha\,,
$$
where in the last inequality we used \eqref{eq.algebra}.
Hence, by Proposition~\ref{p.udc} and \eqref{eq.alphal1infty}, 
the absolute value of the integral on the right hand side 
in \eqref{eq.xxxx}
is bounded from above by a (uniform) constant times $\rho^{v_4-v_3}\le\rho^{j^\iota}$. 
By the definition of $|E_{v_3}|$, we get that $|x_{v_3}(\om)|\rho^{j^\iota}\le\rho^{j^\iota/2}$, 
for all $\om\in\PP_{v_3}\setminus E_{v_3}$.
The second and last term in the sum on the right hand 
side of \eqref{eq.xxxx} decays exponentially fast in $v_3$. 
Altogether, we conclude that $|\int\prod_{\ell=1}^4\xi_{v_\ell}da|$ is decaying 
stretched exponentially fast in $j$ whenever $(v_1,...,v_4)\in S$ and $v_4-v_3\ge j^\iota$.

The case when $v_2-v_1\ge j^\iota$ is easier. Instead of considering the 
functions $L_\ell$, we can apply directly Lemma~\ref{l.switch} with $\ell_0=4$.  
Then, a similar 
reasoning gives also 
the stretched exponential decay of $|\int\prod_{\ell=1}^4\xi_{v_\ell}da|$ in this case. 
Altogether, recalling \eqref{eq.SSS} 
and observing that $|S|$ is growing only polynomially fast in $j$, for each $\iota>0$ 
we find a constant $C$ so that
\begin{equation}
\label{eq.xi4total}
Ew_j^4\le C j^{4/3+2\iota}\,,\qquad\forall j\ge1\,.
\end{equation}

Regarding the term $E w_j^2w_k^2$, we can assume that $k\ge j+2$ since for $k=j+1$ we just can 
apply Cauchy's inequality and \eqref{eq.xi4total} 
for estimating $Ew_j^2w_{j+1}^2$ and Proposition~\ref{p.mainestimate} for 
estimating $Ew_{j}^2Ew_{j+1}^2$ which yields the upper bound $j^{4/3+2\iota}$ for 
$|Ew_j^2w_{j+1}^2-Ew_j^2Ew_{j+1}^2|$.
(For the other terms we have to give a better bound otherwise the bound we get when 
summing over $k$ is not good enough.)
Henceforth, let $k\ge j+2$. We first give a good upper bound for 
$|Ey_j^2w_k^2-Ey_j^2Ew_k^2|$.
This is convenient, since $y_j$ is constant on elements of the partition $\PP_{m-m^{1/4}}$, 
where $m$ denotes the smallest integer in $I_k$.
Let $v=m-m^{1/4}$. 
Since $d_v:=\lambda_0^{-m^{1/4}}$ is decaying stretched exponentially fast in $v$, we can 
apply Lemma~\ref{l.goodpartition} and we find an exceptional set 
$E_v\subset\PP_v$ so that $|E_v|\le Cd_v^{1/2}$ and 
$|x_v(\om)|\ge d_v$ for all $\om\in\PP_v\setminus E_v$.
Let $\PP_v^*$ be a refinement of the partition $\PP_v\setminus E_v$ so that 
for $\om\in\PP_v^*$ we have $d_v=\lambda_0^{-m^{1/4}}\le|x_v(\om)|\le |I_k|^{-3/\alpha}$. 
Applying the local 
estimate~\eqref{eq.mainestimatelocal} in Proposition~\ref{p.mainestimate}, we obtain
$$
\Big|E\Big(w_k(a)^2\mid\{a\in\om\}\Big)-|I_k|\Big|\le C\,,\qquad\forall\om\in\PP_v^*\,.
$$
Recall that $y_j$ is constant on elements of $\PP_v^*$. 
We get
\begin{align*}
Ey_j^2w_k^2&=\sum_{\om\in\PP_v^*}y_j^2(\om)|\om|
E\Big(w_k(a)^2\mid\{a\in\om\}\Big)/\epsilon
+O(\lambda_0^{-m^{1/4}/2}\|y_j^2w_k^2\|_\infty )\\
&\le Ey_j^2(|I_k|+C)+O(\lambda_0^{-m^{1/4}/2}k^{8/3} )\,.
\end{align*}
On the other hand, by the global estimate \eqref{eq.mainestimate} in 
Proposition~\ref{p.mainestimate}, we have $Ey_j^2Ew_k^2\ge Ey_j^2(|I_k|-C)$.
Altogether, we derive
\begin{equation}
\label{eq.55}
|Ey_j^2w_k^2-Ey_j^2Ew_k^2|\le CEy_j^2\,.
\end{equation}
Writing $E|w_j^2-y_j^2|=E|w_j+y_j||w_j-y_j|\le Cj^{2/3}E|w_j-y_j|$, by 
\eqref{eq.chia1} and \eqref{eq.chiahelp}, we derive that $E|w_j^2-y_j^2|$ 
is stretched exponentially decreasing in $j$. 
Hence, by \eqref{eq.55} and once more by \eqref{eq.mainestimate}, it follows
\begin{equation}
\label{eq.43}
|Ew_j^2w_k^2-Ew_j^2Ew_k^2|\le C|I_j|\le Cj^{2/3}\,.
\end{equation}
Recalling \eqref{eq.galkoksmapre}, we can now easily derive \eqref{eq.galkoksma}. This 
concludes the proof of Lemma~\ref{l.LLNy}.
\end{proof}

\subsection{Martingale representation and embedding procedure}

In this section we will follow closely Sections~3.4 and 3.5 in \cite{ps}. 
Let $\LL_j$, $j\ge1$, be the $\sigma$-field generated by $(y_1,y_2,...,y_j)$, and set 
$$
u_j=\sum_{k\ge0} E(y_{j+k}\mid\LL_{j-1})\,.
$$
Then $\{Y_j,\LL_j\}$ defined by $Y_j=y_j+u_{j+1}-u_j$ is a martingale difference sequence. 
Recalling the definition of the $\sigma$-fields $\FF_i$ in the beginning 
of Section~\ref{ss.blocks}, we see that $\LL_{j-1}\subset\FF_{i(j)}$ 
where $i(j)=\max\{i\in I_{j-1}\}$. Hence, we can write
$$
u_j=\sum_{k\ge1}E(E(\xi_{i(j)+k}\mid\FF_{i(j)})\mid\LL_{j-1})\,.
$$
Recall \eqref{eq.chia2} in Lemma~\ref{l.chiapprox} and the to it related notations. 
Recall also that $i(j)\le Cj^{5/3}$ (see, e.g., \eqref{eq.MN}).
Setting 
$E_j=\cup_{k\ge0} E_{i(j),k}$,
we have $|E_j|\le C\rho_0^{i(j)^{\delta}/2}\le C^2e^{-j^{5\delta/3}/C}$,
and there exists a constant $C$ so that 
for a.e. $a\in[0,\epsilon]\setminus E_j$ we have
\begin{equation}
\label{eq.uj1}
|u_j(a)|\le C j^{5\delta/3}\,.
\end{equation}
Further, for $\ell\ge0$, we derive
\begin{equation}
\label{eq.uj2}
|u_j(a)|\le\max(Cj^{5\delta/3},C\ell)\,,
\end{equation}
for a.e. $a\in E_{i(j),\ell}\setminus\cup_{k>\ell}E_{i(j),k}$.
Since 
\begin{equation}
\label{eq.Eij}
|E_{i(j),\ell}|\le\left\{\begin{array}{ll}C\rho_0^{(\ell-i(j)^{\delta})/2}\le C\rho_0^{\ell/4}\,,&
\text{if }\ell\ge2i(j)^{\delta}\ge j^{5\delta/3}/C\\
0\,,&\text{otherwise,}\end{array}\right.
\end{equation}
for $\delta>0$ sufficiently small,
we see that $|\{|u_{M+1}|\ge N^{\gamma}\}|$ is 
summable over $N\ge1$ (recall \eqref{eq.MN}). We conclude 
that, for a.e. $a\in[0,\epsilon]$, there exists a constant $C$ so that
\begin{equation}
\label{eq.yjYj}
\Big|\sum_{j=1}^My_j(a)-Y_j(a)\Big|=|u_{M+1}-u_1|\le CN^\gamma\,.
\end{equation}
In other words, in the following we can work with the martingale 
difference sequence $Y_j$ instead of $y_j$.
The $Y_j$ inherit the law of large numbers shown for $y_j$:

\begin{lemma}
\label{l.LLNY}
For a.e. $a\in[0,\epsilon]$, there exists a constant $C$ so that 
\begin{equation}
\label{eq.Yj2}
\Big|N-\sum_{j=1}^MY_j^2(a)\Big|\le CN^{2\gamma}\,,\qquad\forall N\ge1,
\end{equation}
(where $\gamma>2/5$ is the error exponent in Theorem~\ref{t.main}).
\end{lemma}

\begin{proof}
Put $v_j=u_j-u_{j+1}$. Since $Y_j^2=y_j^2-2y_jv_j+v_j^2$, by Lemma~\ref{l.LLNy} and 
Cauchy's inequality, it is sufficient to show that, 
for a.e. $a\in[0,\epsilon]$, there exists a constant $C$ so that 
\begin{equation}
\label{eq.llnY}
\sum_{j=1}^Mv_j^2\le CN^{4\gamma-1}\,.
\end{equation}
Observe that by \eqref{eq.MN} we have $N^{4\gamma-1}/M\ge CN^{4\gamma-8/5}$, 
and since $\gamma>2/5$ we have $4\gamma-8/5>0$.
By \eqref{eq.uj1} and \eqref{eq.uj2}, there exists a constant $C$ so that 
for all $\delta>0$ sufficiently small and all $M$ sufficiently large we have, 
for $1\le j\le M$,
\begin{equation}
\label{eq.vj1}
v_j(a)^2\le C j^{10\delta/3}\le N^{4\gamma-1}/M\,,
\qquad\text{for a.e. $a\in[0,\epsilon]\setminus (E_j\cup E_{j+1})$},
\end{equation}
and, for all $\ell\ge0$ and $1\le j\le M$, we have
\begin{equation}
\label{eq.vj2}
v_j(a)^2\le\max(Cj^{10\delta/3},C\ell^2)\le\max(N^{4\gamma-1}/M,C\ell^2)\,,
\end{equation}
for a.e. $a\in (E_{i(j),\ell}\cup E_{i(j+1),\ell})\setminus(\cup_{k>\ell}E_{i(j),k}\cup E_{i(j+1),k})$.
Combined with \eqref{eq.Eij}, we get
\begin{align*}
|\{a\in[0,\epsilon]\mid\sum_{j=1}^Mv_j^2\le N^{4\gamma-1}\}|
\le N^{-(4\gamma-1)} \sum_{j=1}^M
\sum_{\ell\ge\sqrt{N^{4\gamma-1}/CM}}
C\ell^2C\rho_0^{\ell/4}\,,
\end{align*}
where the right hand side is summable in $N$. 
This concludes the proof of \eqref{eq.llnY} and, thus, the proof of the lemma.
\end{proof}

\begin{lemma}
\label{l.LLNR}
For a.e. $a\in[0,\epsilon]$, there exists a constant $C$ so that 
\begin{equation}
\label{eq.Rj2}
\Big|\sum_{j=1}^ME(Y_j^2\mid \LL_{j-1})-Y_j^2(a)\Big|\le CN^{2\gamma}\,,\qquad\forall N\ge1,
\end{equation}
(where $\gamma>2/5$ is the error exponent in Theorem~\ref{t.main}).
\end{lemma}

\begin{proof}
Set $R_j=Y_j^2-E(Y_j^2\mid\LL_{j-1})$ and observe that $\{R_j,\LL_j\}$ is a 
martingale difference sequence. By the definition of $Y_j$ and by Minkowski's inequality, 
we have
$$
ER_j^2\le4EY_j^4\le C(Ew_j^4+E|w_j^4-y_j^4|+Ev_j^4)\,.
$$
By \eqref{eq.xi4total}, for all $\iota>0$ we find a constant $C$ so that 
$Ew_j^4\le Cj^{4/3+\iota}$. Since $w_j^4-y_j^4=(w_j^2+y_j^2)(w_j+y_j)(w_j-y_j)$, 
we can apply \eqref{eq.chiahelp} and we derive that $E|w_j^4-y_j^4|$ is uniformly bounded in $j$.
By \eqref{eq.uj1} and \eqref{eq.uj2}, we derive that $Ev_j^4\le Cj^{4\delta}$. 
Hence, for all $\iota>0$, we have 
$$
\sum_{j\ge1}j^{-7/3-\iota}ER_j^2<\infty\,,
$$
and by a martingale result (see, e.g., \cite{chow}) we get that 
$\sum_{j\ge1}j^{-7/6-\iota}R_j$ converges almost surely. By Kronecker's Lemma we conclude 
that, for a.e. $a\in[0,\epsilon]$, there exists a constant $C$ so that
$$
\sum_{j=1}^MR_j\le CM^{7/6+\iota}\le C^2N^{21/30+\iota}\,,
$$
where we used \eqref{eq.MN} in the last inequality. Since $21/30<4/5<2\gamma$ this 
concludes the proof of the lemma.
\end{proof}

Now we apply the following martingale embedding result to the martingale difference 
sequence $Y_j$. For a proof see, e.g., \cite[Theorem~A.1]{hh}.

\begin{theorem}[Skorokhod's representation theorem] 
\label{t.skorokhod}
Let $\{ \sum_{j=1}^MY_j,\ \LL_M,\ M\ge1\}$ 
be a zero-mean, square-integrable martingale. 
Then there exists a probability space supporting a 
zero-mean, square-integrable martingale $\{ \sum_{j=1}^M\widetilde Y_j,\ \widetilde\LL_M,\ M\ge1\}$,
a Brownian motion $W$, and a sequence of nonnegative variables  $T_j$, $j\ge1$, 
such that 
\begin{itemize}
\item[(i)]
$\{ Y_j \}_{j\ge1}$ and $\{ \widetilde Y_j \}_{j\ge1}$ have the same distribution;
\item[(ii)]
$\sum_{j=1}^M \widetilde Y_j = W(\sum_{j=1}^M T_j)$ almost surely;
\item[(iii)]
$E(T_j\mid\GG_{j-1})=E(\widetilde Y_j^2\mid\GG_{j-1})$ almost surely, where 
$\GG_j$ is the $\sigma$-field generated by $\{W(t),\ 0\le t\le\sum_{\ell\le j}T_\ell\}$.
\end{itemize}
\end{theorem}

We will keep the same notation, i.e., instead of writing $\widetilde Y_j$ and $\widetilde\LL_j$, 
we keep writing $Y_j$ and $\LL_j$. Since $\LL_j\subset\GG_j$, for all $j\ge1$, we have 
\begin{equation}
\label{eq.Tjhelp}
E(T_j\mid\GG_{j-1})=E(Y_j^2\mid\GG_{j-1})=E(Y_j^2\mid\LL_{j-1})\,,
\end{equation}
almost surely. We can now show a strong law of large numbers for the sequence $T_j$.

\begin{lemma}
\label{l.LLNT}
For a.e. $a\in[0,\epsilon]$, there exists a constant $C$ so that 
\begin{equation}
\label{eq.Tj}
\Big|N-\sum_{j=1}^MT_j\Big|\le CN^{2\gamma}\,,\qquad\forall N\ge1,
\end{equation}
(where $\gamma>2/5$ is the error exponent in Theorem~\ref{t.main}).
\end{lemma}

\begin{proof}
By \eqref{eq.Tjhelp} we get
$$
N-\sum_{j=1}^MT_j=\Big[N-\sum_{j=1}^MY_j^2\Big]
+\sum_{j=1}^M[Y_j^2-E(Y_j^2\mid\LL_{j-1})]
+\sum_{j=1}^M[E(T_j\mid\GG_{j-1})-T_j]\,,
$$
almost surely. By Lemma~\ref{l.LLNY} and Lemma~\ref{l.LLNR}, the first two terms 
are almost surely bounded by a constant times $N^{2\gamma}$. 
Write $R_j=E(T_j\mid\GG_{j-1})-T_j$. By \eqref{eq.Tjhelp}, $\{R_j,\GG_j\}$ is a martingale 
difference sequence satisfying $ER_j^2\le4EY_j^4$. Hence, we can go along the proof of
Lemma~\ref{l.LLNR} and we get the same upper bound for this term.
\end{proof}

Now we can go word by word along the proof of \cite[Lemma~3.5.3]{ps}
replacing $1/2-\alpha/2+\gamma$ and Lemma~3.5.1 
therein by $\gamma$ and Lemma~\ref{l.LLNT} from our setting, respectively, and we obtain
$$
\Big|\sum_{j=1}^MY_j-W(N)\Big|=O(N^\gamma)\,,\qquad\text{almost surely.}
$$
Recalling \eqref{eq.lastterm} and \eqref{eq.yjYj}, this concludes the proof of Theorem~\ref{t.main}.

\bibliographystyle{plain}
\bibliography{lil}
\end{document}